\documentclass[a4paper,oneside,12pt]{scrartcl}
\usepackage[english]{babel}
\usepackage[latin1]{inputenc}
\usepackage{graphicx}
\usepackage{amssymb}
\usepackage{amsthm}
\usepackage{amsmath}
\usepackage{fancyhdr}
\usepackage{ifthen}
\usepackage{amsfonts}
\usepackage{comment}
\usepackage{multicol}
\usepackage{xspace}

\usepackage{color}
\usepackage{esvect}
\usepackage{stackrel}

\newcommand{\mathCommandFont}[1]{\mathrm{#1}}
\newcommand{\logicFont}[1]{\mathcal{#1}}
\newcommand{\classFont}[1]{\mathrm{#1}}
\newcommand{\problemFont}[1]{\textsc{#1}}

\newcommand{\halfliteral}[1]{\protect\ensuremath{#1}}
\newcommand{\literal}[1]{\halfliteral{#1}\xspace}
\newcommand{\bracketOperator}[3]{\literal{#1\nobreak#3\nobreak#2}}
\newcommand{\commandOperator}[2]{\literal{\mathord{\mathCommandFont{#1}\ifthenelse{\equal{#2}{}}{}{(\nobreak#2\nobreak)}}}}

\newcommand{\set}[3][]{\literal{\left\{#2\;\middle|\;\ifthenelse{\equal{#1}{}}{\text{#3}}{\parbox{#1}{#3}}\right\}}}

\newcommand{\dfn}{\mathrel{\mathop:}=}
\newcommand{\ddfn}{\mathrel{\mathop{{\mathop:}{\mathop:}}}=}

\newcommand{\dep}[1][\cdot]{\literal{\mathCommandFont{=}\ifthenelse{\equal{#1}{}}{}{(\nobreak#1\nobreak)}}}

\newcommand{\poc}[2][]{\bracketOperator{\left\lgroup}{\right\rgroup}{\ifthenelse{\equal{#1}{}}{#2}{\begin{array}{@{}#1@{}}#2\end{array}}}}

\newcommand{\logic}[1]{\literal{\logicFont{#1}}}
\newcommand{\paraLogic}[2]{\ensuremath{\logic{#1}\ifthenelse{\equal{#2}{}}{}{(#2)}}\xspace}
\newcommand{\D}{\logic{D}}
  \newcommand{\df}{\D}
\newcommand{\ff}{\logic{F}}
  \xspaceaddexceptions{\ff}

\newcommand{\ifl}{\logic{IF}}

\newcommand{\bd}{\logic{BD}}
\newcommand{\abd}{\logic{\forall{}\nobreak\text{-}\nobreak{}BD}} 
\newcommand{\rbd}{\logic{RBD}} 
\newcommand{\bbd}{\logic{BBD}} 
\newcommand{\POC}{\logic{POC}}

\newcommand{\pocfo}{\logic{\POC[FO]}}

\newcommand{\pocsvfo}{\logic{\mathsf{D}[FO]}}
\newcommand{\mnfor}{N_\pi[\logic{FO}_r]}
\newcommand{\fopoc}{\logic{FO(\POC)}}

\newcommand{\fopocsv}{\logic{FO(\mathsf{D})}}
\newcommand{\fopocp}{\logic{FO(\POC^+)}}
\newcommand{\fopocpsv}{\logic{FO(\mathsf{D}^+)}}
\newcommand{\pocqf}{\logic{\POC[QF]}}
\newcommand{\pocsvqf}{\logic{\mathsf{D}[QF]}}

\newcommand{\eso}{\logic{ESO}}

\newcommand{\fo}{\logic{FO}}

  \newcommand{\LL}[1][]{\paraLogic{L}{#1}}

\newcommand{\class}[1]{\literal{\classFont{#1}}}

\newcommand{\TIME}{\class{TIME}}
\newcommand{\NTIME}{\class{NTIME}}
\newcommand{\ATIME}{\class{ATIME}}
\newcommand{\SPACE}{\class{SPACE}}

\newcommand{\NSPACE}{\class{NSPACE}}
\newcommand{\ASPACE}{\class{ASPACE}}

\newcommand{\NP}{\class{NP}}

\xspaceaddexceptions{\TIME \NTIME \ATIME \SPACE \NSPACE \ASPACE}

\newcommand{\problem}[1]{\literal{\problemFont{#1}}}

\newcommand{\hsat}[1][]{\nobreak\textnormal{-}\allowbreak\sat\nobreak\ifthenelse{\equal{#1}{}}{}{_{\mathnormal{#1}}}\xspace}
\newcommand{\hmc}[1][]{\nobreak\textnormal{-}\allowbreak\problem{MC}\nobreak\ifthenelse{\equal{#1}{}}{}{_{\mathnormal{#1}}}\xspace}
  \xspaceaddexceptions{\hsat \hmc}

\newcommand{\sat}[1][]{\commandOperator{\problem{Sat}}{#1}}

\newcommand{\simp}{\literal{\Rrightarrow}}

\newcommand{\efgmnr}{\literal{N_\pi\mathrm{EF}_r}}
\newcommand{\arraycomment}[1]{\textnormal{\footnotesize\textsl{#1}}}

\newcommand{\mA}{\mathfrak{A}}
\newcommand{\mB}{\mathfrak{B}}

\newcommand{\mi}{\mathit}

\newcommand{\Str}{\mathrm{Str}}

\newcommand{\dom}{\mathrm{dom}}
\newcommand{\fr}{\mathrm{fr}}

\newcommand{\dl}{\df}

\theoremstyle{plain}
\newtheorem{theorem}{Theorem}[section]
\newtheorem{lemma}[theorem]{Lemma}
\newtheorem{proposition}[theorem]{Proposition}
\newtheorem{corollary}[theorem]{Corollary}
\newtheorem{remark}[theorem]{Remark}
\newtheorem{definition}[theorem]{Definition}

\newcommand{\Title}{Boolean Dependence Logic and Partially-Ordered Connectives}

\newcommand\Author{Johannes Ebbing\footnote{
Leibniz University Hannover,
Theoretical Computer Science,
\{ebbing,lohmann\}@thi.uni-hannover.de}
\and
Lauri Hella\footnote{University of Tampere,
Mathematics,
\{jonni.virtema,lauri.hella\}@uta.fi}
\and
Peter Lohmann\footnotemark[2]
\and
Jonni Virtema\footnotemark[3]\\
}

\newcommand{\PDFAuthor}{Johannes Ebbing, Lauri Hella, Peter Lohmann, Jonni Virtema}

\newcommand{\Keywords}{dependence logic, partially-ordered connectives, expressivity, existential second-order logic}

\usepackage[
  colorlinks=false,pdfdisplaydoctitle=false,pdfkeywords={\Keywords},pdftitle={\Title},pdfauthor={\PDFAuthor}
]{hyperref}

\title{\Title\footnote{Supported by a DAAD grant 50740539, University of Tampere reseach grant, Finnish Academy of Science and Letters: Väisälä fund research grant and grants 266260 and 138163 of the Academy of Finland}}
\author{\Author}

\begin{document}

\maketitle

\begin{abstract}
We introduce a new variant of dependence logic (\df) called Boolean dependence logic (\bd).
In \bd dependence atoms are of the type $\dep[x_1,\dots,x_{n},\alpha]$, where $\alpha$ is a Boolean variable. Intuitively, with Boolean dependence atoms one can express quantification of relations, while standard dependence atoms express quantification over functions.

We compare the expressive power of \bd to \D and first-order logic enriched by partially-ordered connectives, \fopoc. We show that the expressive power of \bd and \df coincide. We define natural syntactic fragments of $\bd$ and show that they coincide with the corresponding fragments of \fopoc with respect to expressive power. We then show that the fragments form a strict hierarchy.
\end{abstract}

\section{Introduction}
Dependence is an important concept in various scientific disciplines. A multitude of formalisms have been designed to model dependences, for example, in database theory, social choice theory, and quantum mechanics. However, for a long time the research has been scattered and the same ideas have been discovered many times over in different fields of science. One important reason, albeit surely not the only one, for this scatteredness was the lack of unified logical background theory for the concept of dependence. Over the last decade the emergence of dependence logic and the extensive and rigorous research conducted on dependence logic and related formalisms have mended this shortcoming.


Dependences between variables in formulae is the most direct way to model dependences in logical systems.
%
%
In first-order logic the order in which quantifiers are written determines dependence relations between variables. For example, when using game 
theoretic semantics to evaluate the formula
\[ 
\forall x_0\exists x_1\forall x_2\exists x_3\, \varphi,
\] 
 the choice for $x_1$ depends on the value for $x_0$, and the choice for $x_3$   depends on the value of both universally quantified variables $x_0$  and  $x_2$.
The first to consider more complex dependences between variables was Henkin \cite{henkin1961} with his partially-ordered quantifiers.
The simplest non-trivial partially-ordered quantifier is usually written in the form
%


\begin{equation}\label{poq}
\poc[c]{\forall x_0 \quad \exists x_1\\ 
\forall x_2\quad \exists x_3}\varphi,
\end{equation}
and the idea is that $x_1$ depends only
on $x_0$ and  $x_3$ depends only on $x_2$. Enderton \cite{Enderton70} and Walkoe \cite{Walkoe70} observed that exactly the properties definable in existential second-order logic (\eso) can be expressed
with partially-ordered quantifiers. Building on the ideas of Henkin, Blass and Gurevich introduced in \cite{blassgure} the narrow Henkin quantifiers
\begin{equation*}\label{poc2}
\poc[cc]{\forall \vec{x}_1 & \exists \alpha_1\\
	\vdots &\vdots \\
	\forall \vec{x}_n & \exists \alpha_n} \varphi.
\end{equation*}
Here $\alpha_1,\ldots,\alpha_n$ are Boolean variables (or, more generally, variables ranging over some fixed finite domains). The idea of Blass and Gurevich was further developed by Sandu and Väänänen in \cite{sava92} where they introduced partially-ordered connectives 
\begin{equation*}\label{poc1}
\poc[cc]{\forall \vec{x}_1 & \bigvee_{b_1\in\{0,1\}}\\
\vdots & \vdots\\
\forall \vec{x}_n & \bigvee_{b_n\in\{0,1\}}}\gamma.
\end{equation*}
Here $\gamma$ is a tuple $(\gamma_{b_1\ldots b_n})_{(b_1,\ldots,b_n)\in\{0,1\}^n}$ of formulae, and the 
choice of each bit $b_i$ determining the disjunct $\gamma_{b_1\ldots b_n}$ to be satisfied depends only on 
$\vec{x}_i$. 

The first to linearize the idea behind the syntax of partially-ordered quantifiers were Hintikka and Sandu \cite{hisa89,hintikka96}, who introduced independence-friendly logic (\ifl). \ifl-logic extends \fo in terms of so-called slashed quantifiers.
%
Dependence logic (\D), introduced by V\"a\"an\"anen \cite{va07},
was inspired by \ifl-logic, but the approach of V\"a\"an\"anen provided a fresh perspective on quantifier dependence. In dependence logic the 
dependence relations between variables are written in terms of novel 
atomic dependence formulae. For example, the partially-ordered quantifier \eqref{poq} can be expressed in dependence logic as follows
\setcounter{equation}{3}
\begin{equation*}
\forall x_0\exists x_1\forall x_2\exists
x_3(\dep[x_2,x_3]\wedge\varphi).
\end{equation*}
The  atomic formula $\dep[x_2,x_3]$ has
the explicit meaning that $x_3$ is completely determined by $x_2$ and nothing else. 

Over the last decade the research related to independence-friendly logic and dependence logic has bloomed. A variety of closely related logics have been defined and various applications suggested, see e.g. journal articles \cite{Abramsky:2007, Abva09, Sevenster:2009, Vaananen+Hodges:2010, Duko12, engko13, grava13, konva13, lohvo13} and conference reports \cite{Bradfield:2005, KoKuLoVi:2011, bradfield13, EHLV2013, EHMMVV13, EKV13}. Furthermore, within the last five years five PhD-thesis have been published on closely related topics, see \cite{nurmithesis, jarmo, gallianithesis, lohmannthesis, ebbingthesis}. See also the monographs \cite{va07, allen:2011}. Research related to partially-ordered connectives has been less active. For recent work, see e.g.~\cite{setu06c, hesetu08, EHLV2013}.


In this article we introduce a new variant of dependence logic called Boolean dependence logic (\bd).
Boolean dependence logic extends first-order logic with special restricted versions of dependence atoms which we call Boolean dependence atoms.
While all variables occurring in dependence atoms
\[
\dep[x_1,\dots,x_n,y]
\]
of dependence logic are first-order variables,
in Boolean dependence atoms
\[
\dep[x_1,\dots,x_{n},\alpha]
\]
of Boolean dependence logic only the antecedents $x_1,\dots,x_{n}$ are first-order variables, whereas the consequent
$\alpha$ is a Boolean variable. A Boolean variable is special kind of variable with values that range over the set $\{\top, \bot\}$, i.e., Boolean variables as assigned a value $\mi{true}$ or $\mi{false}$.

Boolean dependence atoms provide a direct way to express partially-ordered connectives in a similar manner as dependence atoms express partially-ordered quantifiers. To make this connection more clear, we define a syntactic variant of the partially-ordered connectives of Sandu and V\"a\"an\"anen \cite{sava92}, closely related to the narrow Henkin quantifier of Blass and Gurevich \cite{blassgure}. We show that our definition of partially-ordered connectives is, in a strong sense, equivalent to that of Sandu and V\"a\"an\"anen. We then establish a novel connection between partially-ordered connectives and Boolean dependence logic.
For example, the partially-ordered connective
\[
	\poc[c]{\forall x \quad \exists \alpha\\
	\forall y \quad \exists \beta} \varphi,
\]
defined with Boolean variables, can be expressed in Boolean dependence logic by the formula
\begin{equation*}
\forall x\exists \alpha\forall y\exists
\beta\big(\dep[y,\beta]\wedge\varphi\big).
\end{equation*}

Intuitively, the dependence atoms of dependence logic express quantification over functions. The meaning of the dependence atom
\[
\dep[x_1,\dots, x_n,y]
\]
is that there exists a $k$-ary function that maps the values of the variables $x_1,\dots,x_n$ to the value of the variable $y$. Analogously, Boolean dependence atoms can be interpreted as expressing quantification of relations or, more precisely, characteristic functions of relations. In this sense, the meaning of the Boolean dependence atom
\[
\dep[x_1,\dots, x_n,\alpha]
\]
is that there exists a characteristic function of an $n$-ary relation that maps the values of the variables $x_1,\dots,x_n$ to the value $\bot$ or $\top$. Since the expressive powers of dependence logic and existential second-order logic coincide, and since in existential second-order logic it is clear that functions and relations are interdefinable, the question arises whether there is any significant difference between dependence logic and Boolean dependence logic. In fact, in terms of expressive power there is no difference, we show that the expressive power of dependence logic and Boolean dependence logic coincide.  

On the other hand, natural fragments of Boolean dependence logic directly correspond to logics enriched with partially-ordered connectives. We show that in terms of expressive power, certain fragments of Boolean dependence logic coincide with natural logics enriched with partially-ordered connectives. 
In addition, we show that these fragments of Boolean dependence logic form a strict hierarchy with respect to expressive power.

Our results
can be seen as a contribution to the analysis of fragments of existential second-order logic. In particular, we are able to separate natural fragments of existential second-order logic. We also give new insight concerning interdefinability of functions and relations in the framework of dependence logic, and henceforth contribute to the basic research of the dependence phenomenon.
\medskip

The structure of this 
paper 
is as follows. In Section \ref{bdl} we give a formal definition of Boolean dependence logic and state some of its elementary properties. In Section \ref{pocdef} we first briefly discuss the origin of partially-ordered connectives. We then give two alternative definitions for partially-ordered connectives, one familiar from the literature, and a syntactic variant that makes the comparison to Boolean dependence logic more straightforward. Finally, we show that, in a rather strong sense, these two definitions are equivalent. In Section \ref{fofbd} we define three natural fragments of Boolean dependence logic. We name them as bounded Boolean dependence logic (\bbd), restricted Boolean dependence logic (\rbd) and universal Boolean dependence logic (\abd). In Section \ref{normalform} we show a normal form for bounded Boolean dependence logic, and in Section \ref{equiv} we use this normal form to show that the expressive power of \bbd, \rbd and \abd coincide with the expressive power of natural logics enriched with partially-ordered connectives, namely \fopocp, \pocfo and \pocqf, respectively. In Section \ref{separations} we show that \bd, \bbd, \rbd and \abd form a strict hierarchy with respect to 
expressive power.

\section{Boolean dependence logic}\label{bdl}
Boolean dependence logic ($\bd$) is a variant of dependence logic in which the consequents of dependence atoms are Boolean variables instead of first-order variables. We denote Boolean variables by the Greek letters $\alpha$ and $\beta$, whereas we use $x,y,z$ to denote first-order variables as usual. We use $\chi$ to denote a variable that is either a first-order variable or a Boolean variable. Tuples of variables are denoted by $\vec{\alpha},\vec{\beta}$, $\vec{x},\vec{y}$ and $\vec{\chi}$, respectively.

We first recall the syntax of dependence logic $\dl$:
\[
\begin{array}{lcl}
\varphi &\ddfn& x_1=x_2 \mid \neg\, x_1=x_2\mid
R(x_1,\dots,x_n)\mid \neg R(x_1,\dots,x_n)  \mid \\
&& \dep[x_1,\dots,x_n, y] \mid(\varphi \vee \varphi) \mid (\varphi \wedge \varphi) \mid\forall x \varphi \mid \exists x \varphi.
\end{array}
\]
We will give the semantics for dependence logic and Boolean dependence logic simultaneously. 
The syntax of Boolean dependence logic is defined as follows:

\begin{definition}
Let $\tau$ be a relational vocabulary. The syntax of Boolean dependence logic $\bd(\tau)$ is defined from $\tau$ by the following grammar:
\[
\begin{array}{lcl}
\varphi &\ddfn& x_1=x_2 \mid \neg\, x_1=x_2\mid \alpha \mid \neg\alpha \mid \dep[x_1,\dots,x_n, \alpha] \mid\\
&& R(x_1,\dots,x_n)\mid \neg R(x_1,\dots,x_n)  \mid \\
&& (\varphi \vee \varphi) \mid (\varphi \wedge \varphi) \mid\forall x \varphi \mid \exists x \varphi\mid \exists \alpha \varphi.
\end{array}
\]
\end{definition}
The semantics for dependence logic and Boolean dependence logic is defined in terms of teams, i.e., sets of assignments. The only difference is that in Boolean dependence logic we are also assigning values for Boolean variables. Hence, for \bd, {\em assignments} over $\mA$
are finite functions that map first-order variables to elements of $A$ and Boolean variables to elements of $\{\bot,\top\}$. We further assume that $A\cap \{\bot,\top\}$ is always empty.
Notice that Boolean variables are never assigned a value from the domain $A$ and first-order variables are never assigned a value $\bot$ or $\top$.
If $s$ is an assignment, $x$ a first-order variable and $a\in A$, we denote by $s(a/x)$ the assignment with domain $\dom (s)\cup \{x\}$ such that
\[
s(a/x)(\chi) =
\left\{
	\begin{array}{ll}
		s(\chi)  & \mbox{if } \chi\neq x \\
		a & \mbox{if } \chi=x.
	\end{array}
\right.
\]
Analogously, if $s$ is an assignment, $\alpha$ a Boolean variable and $a\in \{\bot,\top\}$, we denote by $s(a/\alpha)$ the assignment with domain $\dom (s)\cup \{\alpha\}$ such that
\[
s(a/\alpha)(\chi) =
\left\{
	\begin{array}{ll}
		s(\chi)  & \mbox{if } \chi\neq \alpha \\
		a & \mbox{if } \chi=\alpha.
	\end{array}
\right.
\]

Let $A$ be a set and $\{x_1,\ldots,x_n,\alpha_1,\dots,\alpha_m\}$ a finite (possibly empty) set  of
variables. 
A {\em team} $X$ of $A$ with domain
\[
\dom (X)=\{x_1,\ldots,x_n,\alpha_1,\dots,\alpha_m\}
\]
is any set of assignments from $\dom(X)$ into $A\cup\{\bot,\top\}$. However, we fix that, for the empty team $\emptyset$, $\dom (\emptyset)=\emptyset$.
If $X$ is a team of $A$, and
$F\colon X\rightarrow A$ and $G\colon X\rightarrow \{\bot,\top\}$ are functions, we use
\begin{itemize}
\item $X(F/x)$ to denote the team $\{s(F(s)/x) \mid s\in X \},$
\item $X(G/\alpha)$ to denote the team $\{s(G(s)/\alpha) \mid s\in X \}$, and
\item $X(A/x)$ to denote the team $\{s (a/x)\mid s\in X\ \textrm{and}\ a\in A \}.$
\end{itemize}
Let $X$ be a team of $\mA$, $W\subseteq \dom (X)$ and $F:X\to A$ a function. We say that the function $F$ is \emph{$W$-determined} if for every assignment $s,s'\in X$ the implication
\[
\forall\, \chi\in W: s(\chi)=s'(\chi) \quad\Rightarrow\quad F(s)=F(s')
\]
holds.
\begin{definition}\label{semantics of bd}
Let $\mA$ be a model and $X$ a team of $\mA$. The satisfaction relation $\mA\models_X\varphi$ for dependence logic and Boolean dependence logic is defined as follows.
%
%
\[
\begin{array}{lcl}
\mA\models_X R(x_1,\dots,x_n) &\Leftrightarrow& \forall s\in X: \big(s(x_1),\dots,s(x_n)\big)\in R^\mA.\\
\mA\models_X \neg R(x_1,\dots,x_n) &\Leftrightarrow& \forall s\in X: \big(s(x_1),\dots,s(x_n)\big)\not\in R^\mA.\\
\mA\models_X (\varphi\wedge\psi) &\Leftrightarrow& \mA\models_X \varphi \text{ and } \mA\models_X \psi.\\
\mA\models_X (\varphi\vee\psi) &\Leftrightarrow& \mA\models_Y \varphi \text{ and } \mA\models_Z \psi,\\
&&\text{for some $Y$ and $Z$ such that $Y\cup Z=X$}.\\
%
%
\mA \models_X \exists x \psi &\Leftrightarrow& \mA \models _{X(F/x)} \psi \text{ for some } F\colon X\to A.\\
\mA \models_X \forall x\psi &\Leftrightarrow& \mA \models _{X(A/x)} \psi.
\end{array}
\]
For dependence logic we have the additional rule:
\begin{align*}
\mA\models_X \dep[x_1,\dots,x_n,y] \quad \Leftrightarrow \quad &\forall s,t\in X: s(x_1)=t(x_1),\dots, s(x_n)=t(x_n)\\
&\text{implies that } s(y)=t(y).
\end{align*}
For Boolean dependence logic we further have the following rules:
\[
\begin{array}{lcl}
\mA\models_X \dep[x_1,\dots,x_n,\alpha] &\Leftrightarrow& \forall s,t\in X: s(x_1)=t(x_1),\dots, s(x_n)=t(x_n)\\
&&\text{implies that } s(\alpha)=t(\alpha).\\
\mA\models_X \alpha &\Leftrightarrow& \forall s\in X: s(\alpha)=\top.\\
\mA\models_X \neg\alpha &\Leftrightarrow& \forall s\in X: s(\alpha)=\bot.\\
\mA \models_X \exists \alpha \psi &\Leftrightarrow& \mA \models _{X(F/\alpha)} \psi \text{ for some } F\colon X\to \{\bot,\top\}.\\
\end{array}
\]
\end{definition}

On the level of sentences the expressive
power of dependence logic coincides with that of existential
second-order logic.
\begin{theorem}\label{d equiv if equiv eso}
$\dl\equiv \eso$.
\end{theorem}
\begin{proof}The fact $\eso\le \df$ is based on the analogous result of \cite{Enderton70,Walkoe70} for partially-ordered quantifiers.  For the converse inclusion, see  \cite{va07} and \cite{Hodges97}.
\end{proof}

\begin{remark}
Note that we do not allow universal quantification of Boolean variables in the syntax of Boolean dependence logic. We have chosen this convention in order to make the comparison to logics with partially-ordered connectives more straightforward. Furthermore, allowing universal quantification of Boolean variables would not add anything essential to Boolean dependence logic, since universal quantification of Boolean variables can be simulated by universal first-order and existential Boolean quantifiers. It is easy to check that, with respect to models with cardinality at least $2$, the formulae $\forall \alpha \varphi$ and 
\[
\forall x \forall y \exists \alpha \Big(\big((x=y \wedge \alpha) \vee (\neg x=y \wedge \neg\alpha)\big)\wedge \varphi\Big),
\]
where $x$ and $y$ are fresh first-order variables that do not occur in $\varphi$, are equivalent.
\end{remark}
%
The set $\fr({\varphi})$ of \emph{free variables} of a $\bd$-formula $\varphi$
is defined recursively in the obvious manner, i.e., in addition to the rules familiar from first-order logic, we have that 
\begin{align*}
%
&\fr\big({\dep[x_1,\dots,x_k,\alpha]}\big)= \{x_1,\dots,x_k,\alpha\},\\
&\fr({\neg\alpha})=\fr({\alpha})=\{\alpha\}, \\
%
%
%
&\fr({\exists \alpha\,\varphi})=\fr({\varphi})\setminus\{\alpha\}.
\end{align*}
If $\fr({\varphi})=\emptyset$, we call $\varphi$ a sentence. We say that the sentence $\varphi$ is true in the model $\mA$ and write
\[
\mA\models\varphi,
\]
if $\mA\models_{\{\emptyset\}}\varphi$ holds. We define the following abbreviation. 
\begin{definition}
Let $V=\{x_{i_1},\dots,x_{i_n}\}$ where $i_j\leq i_{j+1}$, for all $j < n$. 
By $\dep[V,\alpha]$ and $\dep[V,y]$ we denote $\dep[x_{i_1},\dots,x_{i_n},\alpha]$ and $\dep[x_{i_1},\dots,x_{i_n},y]$, respectively.
\end{definition}
Note that while the precise ordering of the $n$ first first-order variables in the dependence atom corresponding to $\dep[V,\alpha]$ or $\dep[V,y]$ is irrelevant, to be precise, we have to fix some ordering. Hence we choose the most canonical one. 

%
We will next state some elementary results on Boolean dependence logic familiar from dependence logic.
In Boolean dependence logic, as well as in dependence logic, the truth of a formula depends only on the interpretations of the variables occurring free in the formula.
\begin{proposition}\label{prop:teamrestriction}
Let $\varphi$ be a $\bd$-formula of vocabulary $\tau$, $\mA$ a $\tau$-model and $X$ a team of $\mA$. If $V\supseteq \fr({\varphi})$, then
\[
\mA\models_X\varphi \quad\text{ iff }\quad \mA\models_{X\upharpoonright V} \varphi.
\]
\end{proposition}
\begin{proof}
The proof is essentially the same as the proof of
the corresponding result for dependence logic,
\cite[Lemma~3.27]{va07},
with just one additional analogous case for the existential Boolean quantifiers.
\end{proof}
Boolean dependence logic is a conservative extension of first-order logic.
\begin{proposition}\label{FO extensionbd}
Let $\varphi$ be a formula of \bd without Boolean variables, i.e.,~$\varphi$ is syntactically a first-order formula. Then for all models $\mA$, teams $X$ of $\mA$ and assignments $s$ of $\mA$:
\begin{enumerate}
\item $\mA\models _{\{s\}}\varphi$ \quad iff\quad $\mA,s\models_{\fo}\varphi$.
\item $\mA\models _X\varphi$ \quad iff\quad $\mA,t\models_{\fo}\varphi$ for every $t \in X$.
\end{enumerate}
\end{proposition}
\begin{proof}
Follows directly from the corresponding result for dependence logic, i.e., 
\cite[Corollary~3.32]{va07}.
\end{proof}
The next two lemmas state that in Boolean dependence logic one can freely substitute subformulae by equivalent formulae and free variables by fresh variables with the same extension.
\begin{lemma}\label{substitution_bd}
Suppose that $\varphi$, $\psi$ and $\vartheta$ are \bd-formulae such that $\fr(\varphi)=\fr(\psi)$ and
$
\varphi\equiv\psi.
$
%
Then
\[
\vartheta \equiv \vartheta(\varphi/\psi),
\]
where $\vartheta(\varphi/\psi)$ is the formula obtained from $\vartheta$ by substituting each occurrence of $\psi$ by $\varphi$.
\end{lemma}
\begin{proof}
The proof is essentially the same as the proof of
the corresponding lemma for dependence logic,
\cite[Lemma 3.25]{va07}.
\end{proof}
\begin{lemma}\label{change_variables}
Let $\varphi$ be a $\bd$-formula and $x_1,\dots,x_n,\alpha_1,\dots,\alpha_m$ the free variables of $\varphi$. Let $y_1,\dots,y_n,\beta_1,\dots,\beta_m$ be distinct variables that do not occur in $\varphi$. If $s$ is an assignment with domain $\{x_1,\dots,x_n,\alpha_1,\dots,\alpha_m\}$, let $s'$ denote the assignment with domain $\{y_1,\dots,y_n,\beta_1,\dots,\beta_m\}$ defined as
\[
s'(\chi) =
\left\{
	\begin{array}{ll}
		s(x_i)  & \mbox{if } \chi= y_i, \\
		s(\alpha_1) & \mbox{if } \chi=\beta_i.
	\end{array}
\right.
\]
If $X$ is a team with domain $\{x_1,\dots,x_n,\alpha_1,\dots,\alpha_m\}$, define that $X'\dfn \{s' \mid s\in X \}$.
Now, for every model $\mA$ and team $X$ of $\mA$ with domain $\fr(\varphi)$ it holds that
\[
\mA\models_X \varphi \quad\text{ iff }\quad \mA\models_{X'}\varphi(y_1/x_1,\dots,y_n/x_n,\beta_1/\alpha_1,\dots,\beta_m/\alpha_m),
\]
where $\varphi(y_1/x_1,\dots,y_n/x_n,\beta_1/\alpha_1,\dots,\beta_m/\alpha_m)$ is obtained from $\varphi$ by substituting each free occurrence of $x_i$ by $y_i$ and $\alpha_j$ by $\beta_j$, $i\leq n$, $j\leq m$.
\end{lemma}
\begin{proof}
Straightforward by Proposition \ref{prop:teamrestriction} and the semantics of Boolean dependence logic.
\end{proof}

\section{Partially-ordered connectives}\label{pocdef}
%
We will first give a short exposition to the origin of partially-ordered connectives. We will then recall the definition of partially-ordered connectives familiar from the literature. After this, we will introduce a notational variant of partially-ordered connectives based on Boolean variables. Finally, we will show that these two variants, in a rather strong sense, coincide.
\subsection{Partially-ordered connectives by Sandu and V\"a\"an\"anen}\label{pocsava}
Partially-ordered connectives were introduced in \cite{sava92} by Sandu and V\"a\"an\"anen. The starting point for their definition was the Henkin quantifier
\[
	\poc[c]{\forall x \quad\exists y\\
	\forall y \quad \exists v},
\]
and the idea to replace the existential quantifiers in the Henkin quantifier by disjunctions. Hence they arrived to the following expression
\[
	\poc[c]{\forall x \quad\bigvee_{i\in\{0,1\}}\\
	\forall y \quad \bigvee_{j\in\{0,1\}}} \big(\varphi_{ij}(x,y)\big)_{i,j\in\{0,1\}},
\]
which they call the partially-ordered connective $D_{1,1}$. The connective $D_{1,1}$ bounds a tuple of formulae of length $4$. The subscript of $D_{1,1}$ reveals the number of rows and the number of universal quantifiers in a given row in the connective. For example, the partially-ordered connective $D_{4,3,3}$ 
is an expression of the form 
\[
	\poc[lc]{\forall x_{1} \forall x_2 \forall x_3 \forall x_4 &\bigvee_{i\in\{0,1\}}\\
	\forall y_{1} \forall y_2 \forall y_3 &\bigvee_{j\in\{0,1\}}\\
	\forall z_{1} \forall z_2 \forall z_3 &\bigvee_{k\in\{0,1\}}},
\]
that binds a tuple of formulae $(\varphi_{ijk}(x_1,x_2,x_3,x_4,y_1,y_2,y_3,z_1,z_2,z_3))_{i,j,k\in\{0,1\}}$. Hence, more generally, a partially-ordered connective, according to Sandu and V\"a\"an\"anen, is an expression of the form 
\[
	D=\poc[cccc]{\forall x_{11}&\dots&\forall x_{1n_1}&\bigvee_{b_1\in\{0,1\}}\\
	\vdots&&\vdots&\vdots\\
	\forall x_{m1}&\dots&\forall x_{mn_m}&\bigvee_{b_m\in\{0,1\}}}
\]
that binds a tuple $\gamma=(\varphi_{\vec{b}})_{\vec{b}\in\{0,1\}^m}$ of formulae\footnote{Sandu and Väänänen consider also even more general versions of partially-ordered connectives in which the indices $b_i$ range over a set of size $k\in \mathbb{N}$. However, these can be easily simulated by the partially-ordered connectives described here.}.
Note that it is assumed that the variables $x_{ij}$ are distinct.
Denoting the tuples
$(x_{i1},\ldots,x_{i{n_i}})$ by $\vec{x}_i$, $1\le i\le m$,
the semantics of the partially-ordered connective $D$ can be written as follows:
\begin{align*}
	\mA,s\models D\,\gamma
	&&\Leftrightarrow&&& \text{there exist functions } g_i:A^{n_i}\to\{0,1\}, 1\leq i \leq m,\\
	&&{}&&&\text{s.t. for all $\vec{a}_i\in A^{n_i}$, $1\leq i \leq m$, } 
	\mA,s'\models\varphi_{\vec{b}}\, \text{, where}\\
	&&{}&&& s'=s(\vec{a}_1/\vec{x}_1,\ldots, \vec{a}_m/\vec{x}_m)
	\text{ and } b_i=g_i(\vec{a}_i),\; 1\le i\le m.
\end{align*}
We denote the set of all such partially-ordered connectives $D$ by $\mathsf{D}$. By \fopocsv we denote the extension of first-order logic by all partially-ordered connectives
$D\in\mathsf{D}$, i.e., the syntax of $\fopocsv(\tau)$ is defined from $\tau$ by the following grammar: 
\[\begin{array}{lcl}
\varphi &::=& x_1=x_2 \mid \neg x_1=x_2\mid R(x_1,\dots,x_n)\mid \neg R(x_1,\dots,x_n) \mid \\
&& (\varphi \vee \varphi) \mid (\varphi \wedge \varphi) \mid D\vec{\varphi}\mid \neg D\vec{\varphi} \mid\forall x \varphi \mid \exists x \varphi
\end{array}\]
where $D\in \mathsf{D}$ and $\vec{\varphi}$ is a vector of formulae of the appropriate length. By $\fopocpsv$, we denote the fragment of $\fopocsv$ that allows only positive occurrences of partially-ordered connectives, i.e., the syntax of $\fopocpsv$ is defined as the syntax for $\fopocsv$ but without the rule $\neg D\vec{\varphi}$. 
Furthermore, by \pocsvfo and \pocsvqf we denote the logics consisting of formulae of the form
\[
D\,(\varphi_{\vec{b}})_{\vec{b}\in\{0,1\}^m},
\]
where $D\in\mathsf{D}$ and
$\varphi_{\vec{b}}$, for every $\vec{b}\in\{0,1\}^m$, is a formula of first-order logic or a quantifier free formula of first-order logic, respectively. The semantics for these logics is defined in terms of models and assignments in the usual way, i.e., in the same manner as for first-order logic, with the additional clause for the partially-ordered connectives $D$.

The expressive power of logics with partially-ordered connectives was studied extensively by
Hella, Sevenster and Tulenheimo in \cite{hesetu08}. They showed that \pocsvqf can be used as
logical formalism for describing \emph{constraint satisfaction problems} (CSP): a natural syntactic restriction to 
the quantifier free formulas leads to a fragment of \pocsvqf that captures exactly the class of
all CSP (with a fixed target structure). Furthermore, another syntactic restriction leads to a fragment
that has the same expressive power as the logic MMSNP. (We refer to \cite{FeVar} for the definitions
of CSP and MMSNP.) The logic \pocsvqf itself was shown to have the same expressive power as
strict NP (SNP).\footnote{Strict NP is also known as strict $\Sigma^1_1$.} 
Strict NP is the fragment of $\eso$ that consists of all formulas of the
form $\exists S_1\ldots\exists S_n\forall x_1\ldots\forall x_m\,\varphi$, where $\varphi$
is a quantifier free first-order formula in a relational vocabulary.

\begin{theorem}[\cite{hesetu08}]\label{pocqf-snp}
$\pocsvqf\equiv\mathrm{SNP}$.
\end{theorem}

In addition, it was shown in \cite{hesetu08} that \fopocsv has the zero-one law.
\emph{Zero-one law} is a property of logics defined as follows. Let $\varphi$ be a $\tau$-sentence
of some logic, where $\tau$ is a relational vocabulary. Let $\Str_n(\tau)$ be the set
of all $\tau$-structures with domain $\{0,\ldots,n-1\}$. The limit probability of $\varphi$ is
\[
 \mu(\varphi):=\lim_{n\to\infty}\frac{|\{\mA\in\Str_n(\tau)\mid \mA\models\varphi\}|}{|\Str_n(\tau)|}.
\]
A logic $\mathcal{L}$ has the zero-one law if $\mu(\varphi)$ exists and is equal to $0$ or $1$ for every
$\mathcal{L}(\tau)$-sentence $\varphi$ in a relational vocabulary $\tau$.
(See \cite{ebbinghaus99} for more information on zero-one laws.)

\begin{theorem}[\cite{hesetu08}]\label{fopoc-01}
\fopocsv has the zero-one law. 
\end{theorem}

%
\subsection{Partially-ordered connectives with Boolean variables}
We will deviate from the definitions of \cite{sava92} in two ways: First, we will replace the 
disjunctions $\bigvee_{b_i\in\{0,1\}}$ by existentially quantified Boolean variables $\exists\alpha_i$;
this makes it easier to relate logics with partially-ordered connectives to fragments of Boolean dependence logic. Second, 
in order to simplify the proofs in Section~\ref{equiv}, we will relax the restriction that the variables 
$x_{ij}$ should be distinct.
%
This approach to partially-ordered connectives is closely related to the narrow Henkin quantifier introduced by Blass and Gurevich \cite{blassgure}.

\begin{definition}
Let $\vec{x}_i=(x_{i1},\ldots,x_{in_i})$, $1\le i\le m$, be tuples of first-order variables, and
let $\alpha_i$, $1\le i\le m$, be distinct Boolean variables. Then
\[
	C=\poc[cccc]{\forall \vec{x}_1&\exists\alpha_1\\
	\vdots&\vdots\\
	\forall \vec{x}_m&\exists\alpha_m}
\]
is a \emph{partially-ordered connective}. The \emph{pattern} of $C$ is
$\pi=(n_1,\ldots,n_m,E)$, where $E$ describes the identities between the 
variables in the tuples $\vec{x}_1,\ldots,\vec{x}_m$, i.e., 
\[
	E=\{(i,j,k,l)\mid 1\le i,k\le m, 1\le j\le n_i,1\le l\le n_k, x_{ij}=x_{kl}\}.
\]
If $C$ is a partially-ordered connective with pattern $\pi$, we denote the connective $C$ 
by $N_\pi \,\vec{x}_1\alpha_1\ldots\vec{x}_m\alpha_m$.
\end{definition}

\begin{definition}
Let $\tau$ be a relational vocabulary.
Syntax of $\fopoc(\tau)$ is defined from $\tau$ by the following grammar:
\[\begin{array}{lcl}
\varphi &::=& \alpha \mid \neg\alpha \mid x_1=x_2 \mid \neg x_1=x_2\mid R(x_1,\dots,x_n)\mid \\
&&  \neg R(x_1,\dots,x_n) \mid(\varphi \vee \varphi) \mid (\varphi \wedge \varphi) \mid\forall x \varphi \mid \exists x \varphi\mid\\
&&N_\pi \,\vec{x}_1\alpha_1\ldots\vec{x}_m\alpha_m\, \varphi\mid 
\neg N_\pi \,\vec{x}_1\alpha_1\ldots\vec{x}_m\alpha_m\, \varphi.
\end{array}\]
For $N_\pi \,\vec{x}_1\alpha_1\ldots\vec{x}_m\alpha_m\, \varphi$ to be a syntactically correct formula, we require that the identities between the variables in $\vec{x}_1,\dots, \vec{x}_m$ are exactly those described in $\pi$. Additionally the Boolean variables $\alpha_1,\dots,\alpha_m$ are all required to be distinct.
\end{definition}
$\fopocp$ is the syntactic fragment of $\fopoc$ that allows only positive occurrences of partially-ordered connectives.
In other words, syntax for $\fopocp$ is defined by the grammar
\[\begin{array}{lcl}
\varphi &::=& \alpha \mid \neg\alpha \mid x_1=x_2 \mid \neg x_1=x_2\mid R(x_1,\dots,x_n)\mid \neg R(x_1,\dots,x_n) \mid\\
&&  (\varphi \vee \varphi) \mid (\varphi \wedge \varphi) \mid\forall x \varphi \mid \exists x \varphi\mid N_\pi \,\vec{x}_1\alpha_1\ldots\vec{x}_m\alpha_m\, \varphi.
\end{array}\]
The fragment $\pocfo$ of $\fopoc$ consists of exactly the formulae of the form
\[
N_\pi \,\vec{x}_1\alpha_1\ldots\vec{x}_m\alpha_m\,\varphi,
\]
where $N_\pi \,\vec{x}_1\alpha_1\ldots\vec{x}_m\alpha_m$ is a partially-ordered connective and $\varphi \in \fo$. Analogously, the logic $\pocqf$ consists of exactly the formulae of the form
\[
N_\pi \,\vec{x}_1\alpha_1\ldots\vec{x}_m\alpha_m\,\varphi,
\]
where $N_\pi \,\vec{x}_1\alpha_1\ldots\vec{x}_m\alpha_m$ is a partially-ordered connective and $\varphi$ is a quantifier free formula of $\fo$.

The semantics for these logics is defined in terms of models and assignments
in the usual way, i.e., in the same manner as for first-order logic. The clause for formulae starting with a partially-ordered connective 
$C=N_\pi \,\vec{x}_1\alpha_1\ldots\vec{x}_m\alpha_m$ with pattern $\pi=(n_1,\ldots,n_m,E)$ is 
the following:
%
\setlength{\jot}{0pt}
\begin{align*}
	\mA,s\models C\, \varphi
	&&\Leftrightarrow&&& \text{there exist functions } f_i:A^{n_i}\to\{\bot,\top\}, 1\le i\le m,\\
	&&{}&&&\text{such that for all tuples }\vec{a}_i\in A^{n_i}, 1\le i\le m: \\
	&&{}&&& \text{if $\vec{a}_1\ldots\vec{a}_m$ is of pattern $\pi$, then }\mA,s'\models\varphi, 
	\text{where}\\
	&&{}&&& s'=s(\vec{a}_1/\vec{x}_1,\ldots, \vec{a}_m/\vec{x}_m,
	f_1(\vec{a}_1)/\alpha_1,\ldots,f_m(\vec{a}_m)/\alpha_m).
\end{align*}
%
Here ``$\vec{a}_1\ldots\vec{a}_m$ is of pattern $\pi$" means that
$a_{ij}=a_{kl}$ whenever $(i,j,k,l)\in E$. 
Note that $s'$ is well-defined for all tuples $\vec{a}_1\ldots\vec{a}_m$ that are of pattern $\pi$.

It is straightforward to prove that the expressive powers of $\fopoc$, $\fopocp$, $\pocfo$ and $\pocqf$ coincide with that of $\fopocsv$, $\fopocpsv$, $\pocsvfo$ and $\pocsvqf$, respectively. However the related proofs are quite technical.
%
\begin{lemma}\label{poctobpoc}
For every formula $\varphi\in\fo(\mathsf{D})$ there exists a formula $\varphi^*\in\fopoc$ such that for every model $\mA$ and assignment $s$ of $\mA$ it holds that
\[
\mA,s\models \varphi \quad\Leftrightarrow\quad \mA,s\models \varphi^*.
\] 
\end{lemma}
\begin{proof}
The claim follows by the following translation $\varphi\mapsto\varphi^*$. For atomic and negated atomic formulas the translation is the identity. For propositional connectives and first-order quantifiers the translation is defined in the obvious inductive way, i.e.,
\begin{eqnarray*}
\neg\varphi  &\mapsto& \neg\varphi^*,\\
(\varphi \wedge \psi)   &\mapsto& (\varphi^* \wedge \psi^*),\\
(\varphi \vee \psi)   &\mapsto& (\varphi^* \vee \psi^*),\\
\exists x\, \varphi   &\mapsto& \exists x\,  \varphi^*,\\
\forall  x\, \varphi   &\mapsto& \forall x\,  \varphi^*.
\end{eqnarray*}
The only nontrivial case is the case for the partially-ordered connectives $D\in\mathsf{D}$.
Let $D\in\mathsf{D}$ be the partially-ordered connective
\[
	\poc[cccc]{\forall \vec{x}_1&\bigvee_{b_1\in\{0,1\}}\\
	\vdots&\vdots\\
	\forall \vec{x}_m&\bigvee_{b_m\in\{0,1\}}}
\]
and let $\gamma=(\varphi_{\vec{b}})_{\vec{b}\in\{0,1\}^m}$ be a tuple of $\fo(\mathsf{D})$-formulae. Let $\pi$ be the pattern that arises from the tuples $\vec{x}_1,\dots,\vec{x}_m$ of variables. We define that
\[
(D\,\gamma)^* \dfn N_\pi \,\vec{x}_1\alpha_1\ldots\vec{x}_m\alpha_m\,\bigwedge_{K\subseteq\{1,\ldots,m\}}\Bigl(\big(\bigwedge_{i\in K}\alpha_i\land
	\bigwedge_{\stackrel{1\leq i\leq m}{i\not\in K}}\neg\alpha_i\big)\to (\varphi_{\vec{b}_K})^*\Bigr),
\]
where $\vec{b}_K=(b_1,\ldots, b_m)\in\{0,1\}^m$ is the
tuple such that $b_i=1\;\Leftrightarrow\; i\in K$.

The claim now follows by a simple induction on the structure of the formulae. The only nontrivial case is the case for partially-ordered connectives. Let $D\in\mathsf{D}$ be the partially-ordered connective
\[
	\poc[cccc]{\forall \vec{x}_1&\bigvee_{b_1\in\{0,1\}}\\
	\vdots&\vdots\\
	\forall \vec{x}_m&\bigvee_{b_m\in\{0,1\}}}
\]
and let $\gamma=(\varphi_{\vec{b}})_{\vec{b}\in\{0,1\}^m}$ be a tuple of $\fopocsv$-formulae. Let $\pi=(n_1,\dots,n_m,E)$ be the pattern that arises from the tuples $\vec{x}_1,\dots,\vec{x}_m$ of variables. Note that since the variables in $\vec{x}_1,\dots,\vec{x}_m$ are all distinct the pattern $\pi$ does not give rise to any nontrivial identities between variables, i.e., other identities than those of the type $x_{ij}=x_{ij}$.
Assume that the claim holds for each formula in the tuple $\gamma$. By the semantics of the partially-ordered connective $D$,
\begin{equation*}
\mA,s\models D\, \gamma
\end{equation*}
if and only if there exists functions
\[
g_i:A^{n_i}\to\{0,1\},
\]
$1\leq i\leq m$, such that for all $\vec{a}_i\in A^{n_i}$, $1\leq i \leq m$, 
\[
\mA,s(\vec{a}_1/\vec{x}_1,\ldots, \vec{a}_m/\vec{x}_m)\models\varphi_{\vec{b}},
\]
where $\vec{b}=\big(g_1(\vec{a}_1),\dots, g_m(\vec{a}_m)\big)$. By the induction hypothesis, this holds if and only if there exists functions
\[
g_i:A^{n_i}\to\{0,1\},
\]
$1\leq i\leq m$, such that for all $\vec{a}_i\in A^{n_i}$, $1\leq i \leq m$, 
\[
\mA,s(\vec{a}_1/\vec{x}_1,\ldots, \vec{a}_m/\vec{x}_m)\models \varphi_{\vec{b}}^*,
\]
where $\vec{b}=\big(g_1(\vec{a}_1),\dots, g_m(\vec{a}_m)\big)$. At this point it is useful to note that functions from $A^{n_i}$ to $\{0,1\}$ and functions from $A^{n_i}$ to $\{\bot,\top\}$ are essentially the same functions. Also note that a tuple $\vec{a}_1\ldots\vec{a}_m$ is always of pattern $\pi$. Hence the above holds if and only if
there exists functions
\[
f_i:A^{n_i}\to\{\bot,\top\},
\]
$1\leq i\leq m$, such that for all $\vec{a}_i\in A^{n_i}$, $1\leq i \leq m$, if $\vec{a}_1\ldots\vec{a}_m$ is of pattern $\pi$, then
\[
\mA,s'\models \bigwedge_{K\subseteq\{1,\ldots,m\}}\Bigl(\big(\bigwedge_{i\in K}\alpha_i\land
	\bigwedge_{\stackrel{1\leq i\leq m}{i\not\in K}}\neg\alpha_i\big)\to (\varphi_{\vec{b}_K})^*\Bigr),
\]
where
\[
s'=s(\vec{a}_1/\vec{x}_1,\ldots, \vec{a}_m/\vec{x}_m,f_1(\vec{a}_1)/\alpha_1,\dots,f_m(\vec{a}_m)/\alpha_m)
\]
and $\vec{b}_K=(b_1,\ldots, b_m)\in\{0,1\}^m$ is the
tuple such that $b_i=1\;\Leftrightarrow\; i\in K$. Furthermore, by the semantics of the partially-ordered connective $N_\pi$, the above holds if and only if
\[
\mA,s\models N_\pi \,\vec{x}_1\alpha_1\ldots\vec{x}_m\alpha_m\,\bigwedge_{K\subseteq\{1,\ldots,m\}}\Bigl(\big(\bigwedge_{i\in K}\alpha_i\land
	\bigwedge_{\stackrel{1\leq i\leq m}{i\not\in K}}\neg\alpha_i\big)\to (\varphi_{\vec{b}_K})^*\Bigr).
\]

\end{proof}
\begin{lemma}\label{bpoctopoc}
For every formula $\varphi\in\fopoc$ without free Boolean variables there exists a formula $\varphi^+\in\fo(\mathsf{D})$ such that for every model $\mA$ and assignment $s$ of $\mA$ it holds that
\[
\mA,s\models \varphi \quad\Leftrightarrow\quad \mA,s\models \varphi^+.
\] 
\end{lemma}
\begin{proof}
The claim follows by the following translation $\varphi\mapsto\varphi^+$. Note that the way we handle partially-ordered connectives ensures that we do not need a clause for Boolean variables nor for negated Boolean variables. For atomic and negated atomic formulae the translation is the identity. For propositional connectives and first-order quantifiers the translation is defined in the obvious inductive way, i.e.,
\begin{eqnarray*}
\neg\varphi  &\mapsto& \neg\varphi^+,\\
(\varphi \wedge \psi)   &\mapsto& (\varphi^+ \wedge \psi^+),\\
(\varphi \vee \psi)   &\mapsto& (\varphi^+ \vee \psi^+),\\
\exists x\, \varphi   &\mapsto& \exists x  \,\varphi^+, \\
\forall  x \,\varphi   &\mapsto& \forall x  \,\varphi^+.
\end{eqnarray*}
The only nontrivial case is the case for the partially-ordered connectives $N_\pi$.
Let $\pi=(n_1,\ldots,n_m,E)$ be a pattern and let
\[
N_\pi \,\vec{x}_1\alpha_1\ldots\vec{x}_m\alpha_m \psi
\]
be an $\fopoc$-formula without free Boolean variables.
Furthermore, let 
$\vec{y}_1,\dots,\vec{y}_m$ be tuples of distinct fresh variables such that $\vec{y}_i=\{y_{i1},\dots,y_{in_i}\}$, for every $i\leq m$.
We define that
\[
(N_\pi \,\vec{x}_1\alpha_1\ldots\vec{x}_m\alpha_m\, \psi)^+ \dfn \poc[cccc]{\forall \vec{y}_1&\bigvee_{b_1\in\{0,1\}}\\
	\vdots&\vdots\\
	\forall \vec{y}_m&\bigvee_{b_m\in\{0,1\}}}\,(\psi^+_{\vec{b}})_{\vec{b}\in\{0,1\}^m},
\]
where, for each $\vec{b}=(b_1,\dots,b_m)\in\{0,1\}^m$, $\psi_{\vec{b}}$ is the
formula 
\[
	\bigl(\bigwedge_{(i,j,k,l)\in E}y_{ij}=y_{kl}\bigr)\to 
	\psi'( \vartheta_{b_1}/\alpha_1,\ldots, \vartheta_{b_m}/\alpha_m),
\]
where $\psi'$ is the formula obtained from $\psi$ by replacing each free occurrence of the variable
$x_{ij}$ by some variable $y_{kl}$ such that $(i,j,k,l)\in E$, $\vartheta_0$ is the formula $\neg (y_{11}=y_{11})$ and $ \vartheta_1$ is the formula $y_{11}=y_{11}$.

The claim now follows by a simple induction on the nesting depth of partially-ordered connectives in formulae, i.e., the highest number of nested partially-ordered connectives in formulae. The case without partially-ordered connectives is trivial. The cases for first-order operations are easy. The only nontrivial case is the case for partially-ordered connectives. Let
\[
N_\pi \,\vec{x}_1\alpha_1\ldots\vec{x}_m\alpha_m
\]
be a partially-ordered connective with pattern $\pi=(n_1,\ldots,n_m,E)$ and let $\psi$ be a \fopoc-formula such that
\[
\varphi\dfn N_\pi \,\vec{x}_1\alpha_1\ldots\vec{x}_m\alpha_m\, \psi
\]
does not have free Boolean variables. Furthermore, assume that the claim holds for every $\fopoc$-formula that does not have free boolean variables and has a lower nesting depth of partially-ordered connectives than $\varphi$. Let 
$\vec{y}_1,\dots,\vec{y}_m$ be tuples of distinct fresh variables not occurring in $\varphi$ such that $\vec{y}_i=\{y_{i1},\dots,y_{in_i}\}$, for every $i\leq m$. Now, by the semantics of the partially-ordered connective $N_\pi$,
\[
\mA,s \models N_\pi \,\vec{x}_1\alpha_1\ldots\vec{x}_m\alpha_m\, \psi
\]
if and only if there exist functions
\[
f_i:A^{n_i}\to\{\bot,\top\},
\]
$1\le i\le m$,
such that for all tuples $\vec{a}_i\in A^{n_i}$, $1\le i\le m$,
if $\vec{a}_1\ldots\vec{a}_m$ is of pattern $\pi$ then
\[
\mA,s'\models  \psi,
\]
where $s'=s(\vec{a}_1/\vec{x}_1,\ldots, \vec{a}_m/\vec{x}_m, f_1(\vec{a}_1)/\alpha_1,\ldots,f_m(\vec{a}_m)/\alpha_m)$.
Clearly this holds if and only if there exist functions
\[
f_i:A^{n_i}\to\{\bot,\top\},
\]
$1\le i\le m$,
such that for all tuples $\vec{a}_i\in A^{n_i}$, $1\le i\le m$,
\[
\mA,t'\models \bigl(\bigwedge_{(i,j,k,l)\in E}y_{ij}=y_{kl}\bigr)\to \psi',
\]
where $t'=s(\vec{a}_1/\vec{y}_1,\ldots, \vec{a}_m/\vec{y}_m, f_1(\vec{a}_1)/\alpha_1,\ldots,f_m(\vec{a}_m)/\alpha_m)$ and $\psi'$ is obtained from $\psi$ by replacing each variable $x_{ij}$ occurring free in $\psi$ by some variable $y_{kl}$ such that $(i,j,k,l)\in E$.
At this point it is helpful to note that functions from $A^{n_i}$ to $\{\bot,\top\}$ and functions from $A^{n_i}$ to $\{0,1\}$ are essentially the same functions. Hence the above holds if and only if there exist functions
\[
g_i:A^{n_i}\to\{0,1\},
\]
$1\le i\le m$,
such that for all tuples $\vec{a}_i\in A^{n_i}$, $1\le i\le m$,
\[
\mA,t\models \bigl(\bigwedge_{(i,j,k,l)\in E}y_{ij}=y_{kl}\bigr)\to 
	\psi'( \vartheta_{b_1}/\alpha_1,\ldots, \vartheta_{b_m}/\alpha_m),
\]
where $t=s(\vec{a}_1/\vec{y}_1,\ldots, \vec{a}_m/\vec{y}_m)$, $b_i=g_i(\vec{a}_i)$ for each $i\leq m$, $\vartheta_0$ denotes the formula $\neg (y_{11}=y_{11})$ and $\vartheta_1$ denotes the formula $y_{11}=y_{11}$.
Remember that for $\vec{b}=(b_1,\dots,b_m)\in \{0,1\}^m$ we defined that
\[
\psi_{\vec{b}}= \bigl(\bigwedge_{(i,j,k,l)\in E}y_{ij}=y_{kl}\bigr)\to 
	\psi'( \vartheta_{b_1}/\alpha_1,\ldots, \vartheta_{b_m}/\alpha_m).
\]
Hence, by the inductive hypothesis, the above holds if and only if there exist functions
\[
g_i:A^{n_i}\to\{0,1\},
\]
$1\le i\le m$,
such that for all tuples $\vec{a}_i\in A^{n_i}$, $1\le i\le m$,
\[
\mA,t\models \psi^+_{\vec{b}}
\]
where $t=s(\vec{a}_1/\vec{y}_1,\ldots, \vec{a}_m/\vec{y}_m)$ and $\vec{b}=\big(g_1(\vec{a}_1),\dots,g_m(\vec{a}_m) \big)$.
Finally, by the semantics of the partially-ordered connectives $\mathsf{D}$, this holds if and only if
\[
\mA,s\models \poc[cccc]{\forall \vec{y}_1&\bigvee_{b_1\in\{0,1\}}\\
	\vdots&\vdots\\
	\forall \vec{y}_m&\bigvee_{b_m\in\{0,1\}}}\,(\psi^+_{\vec{b}})_{\vec{b}\in\{0,1\}^m}.
\]
\end{proof}
\begin{proposition}\label{fopocsv}
$\fopoc\equiv\fopocsv$.
\end{proposition}
\begin{proof}
Follows directly from Lemmas \ref{poctobpoc} and \ref{bpoctopoc}.
\end{proof}
The translations defined in Lemmas \ref{poctobpoc} and \ref{bpoctopoc} directly yield the following equivalences between fragments of $\fopoc$ and $\fopocsv$ as well.
\begin{proposition}\label{fopocsvf}
The following equivalences hold: $\fopocpsv\equiv\fopocp$, $\pocfo\equiv\pocsvfo$ and $\pocqf\equiv\pocsvqf$.
\end{proposition}

By Theorem \ref{fopoc-01}, \fopocsv has the zero-one law. Thus,
by Proposition~\ref{fopocsv}, the same is true for \fopoc and all its fragments as well.
\begin{corollary}\label{0-1-law}
\fopoc has the zero-one law.
Therefore,  \fopocp, \pocfo and \pocqf also have the zero-one law.
\end{corollary}

We will also make use of the fact that, by Theorem \ref{pocqf-snp}, $\pocqf\equiv\pocsvqf$ is equivalent to strict NP:

\begin{corollary}\label{abd-pocqf}
$\pocqf\equiv\mathrm{SNP}$.
\end{corollary}
%

%

\section{Fragments of Boolean dependence logic}\label{fofbd}
%
%
%
In this section we define fragments of Boolean dependence logic that correspond to the fragments $\fopocp$, $\pocfo$ and $\pocqf$ of $\fopoc$, with respect to expressive power. We restrict our attention to sentences, i.e., to formulae without free variables.

We will first demonstrate an exemplary translation from $\fopocp$ to Boolean dependence logic in order to visualize the fragments of $\bd$ that correspond to $\fopocp$, $\pocfo$ and $\pocqf$, respectively.
Let $\varphi$ be the $\fopocp$-sentence
\[
\forall x_0 \forall x_1\exists x_2 \poc[cc]{\forall\vec{y} &\exists \alpha\\
\forall \vec{z} &\exists \beta}\psi,
\]
where $\psi$ is syntactically first-order. Clearly $\varphi$ is equivalent to the sentence
\[
\vartheta\dfn\forall x_0 \forall x_1\exists x_2 \forall \vec{y}\, \forall \vec{z}\,\exists \alpha\exists\beta\big(\dep[x_0,x_1,x_2,\vec{y},\alpha]\wedge \dep[x_0,x_1,x_2,\vec{z},\beta]\wedge \psi\big)
\]
of Boolean dependence logic. After examining the syntactic form of $\vartheta$ we notice some regularity in the Boolean dependence atoms of $\vartheta$. The existential first-order quantifier $\exists x_2$ seems to partition the set of first-order variables quantified in $\vartheta$ into two parts. Every variable before the quantifier $\exists x_2$ including the variable $x_2$ itself are in the antecedent of every Boolean dependence atom in $\vartheta$. The variables after $\exists x_2$ are not in the antecedent of every Boolean dependence atom of $\vartheta$. This observation leads us to the following somewhat technical definitions.
\begin{definition}\label{occurrence}
Let $\varphi$ be a formula of $\bd$ or $\fopoc$, $\psi$ a subformula of $\varphi$ and $n\in\mathbb{N}$. Note that we may consider formulae as just strings of symbols of some prescribed vocabulary. Hence, we may order the occurrences of subformulae of a given formula. By $[\psi,\varphi]_n$ we denote the $n$th occurrence of the formula $\psi$ in $\varphi$. If there is just one occurrence of $\psi$ in $\varphi$, we may drop the subscript and just write $[\psi,\varphi]$.
\end{definition}
\begin{definition}\label{dependenceset}
Let $\varphi$ be a formula of $\bd$ or $\fopoc$, and $[\psi,\varphi]_n$ an occurrence of the subformula $\psi$ of $\varphi$. We define $V([\psi,\varphi]_n)$ to be the set of all first-order variables $x$ such that $x$ is free in $\varphi$ or
the occurrence $[\psi, \varphi]_n$ of $\psi$ is in a scope of $\forall x$ or $\exists x$ in $\varphi$.
\end{definition}
Hence, for sentences $V([\psi,\varphi]_n)$ is the set of quantified variables that dominate $[\psi,\varphi]_n$ in the syntactic tree of $\varphi$. We are now ready to define the fragments of $\bd$ that correspond to $\fopocp$, $\pocfo$ and $\pocqf$, respectively.
\begin{definition}\label{booleanlogics}
We define the following syntactic fragments of Boolean dependence logic.
\begin{enumerate}
\item \emph{Bounded Boolean dependence logic}, \bbd, is the syntactic restriction of \bd to formulae $\varphi$ such that the following condition holds.

If $[\exists x \psi,\varphi]_n$ is an occurrence of a subformula $\exists x \psi$ of $\varphi$ and a dependence atom $\dep[x_1,\dots,x_n,\alpha]$ is a subformula of $\psi$ then
\[
V([\psi,\varphi]_t)\subseteq \{x_1,\dots,x_n\},
\]
where $[\psi,\varphi]_t$ is the occurrence of $\psi$ that occurs in $[\exists x \psi,\varphi]_n$.
\item \emph{Restricted Boolean dependence logic}, \rbd, is the restriction of \bd to formulae where no dependence atoms occur inside the scope of an existential first-order quantifier.
\item \emph{Universal Boolean dependence logic}, \abd, is the restriction of \bd to formulae without existential quantification of first-order variables.
\end{enumerate}
\end{definition}
In Section \ref{equiv} we establish that the above definitions indeed fit for our purposes. We show that
\[
\abd\equiv \pocqf,\, \rbd\equiv \pocfo\, \text{ and }\, \bbd\equiv \fopocp.
\] 
It is easy to see, that every $\abd$ formula is an $\rbd$ formula, every $\rbd$ formula is a $\bbd$ formula and every $\bbd$ formula is a $\bd$ formula. Hence, once we realize that every $\bd$-sentence can be simulated by a sentence of dependence logic, we obtain the following hierarchy.
\begin{proposition}\label{logic-inclusions}
$\abd\leq\rbd\leq\bbd\leq\bd\leq\df$.
\end{proposition}
\begin{proof}
The first three inclusions follow by the observation made above.
For the last inclusion we give a translation $\varphi\mapsto \varphi^+$ that maps $\bd$-sentences to equivalent $\dl$ sentences.
We will establish that for every $\bd$-sentence $\varphi$ and every structure $\mA$ it holds that
\begin{equation}\label{bdtod}
\mA \models \varphi \quad\text{ iff }\quad \mA \models \varphi^\mi{+}.
\end{equation}
For each Boolean variable $\alpha$ and for the symbols $\bot$ and $\top$, we introduce distinct fresh first-order variables $x_\alpha$, $x_\bot$ and $x_\top$. Without loss of generality, we may assume that these first-order variables do not appear in the formulae of Boolean dependence logic.
We may, without loss of generality, restrict our attention to models with at least two elements. For a sentence $\varphi\in \bd$ we define that 
\[
\varphi^\mi{+} \dfn \exists x_\bot \exists x_\top (x_\bot \neq x_\top \land \varphi^*),
\]
where $\varphi^*$ is the sentence obtained from $\varphi$ by the following recursive translation. For first-order literals, the translation is the identity. The remaining clauses are as follows:
\begin{eqnarray*}
\dep[x_1,\dots,x_k,\alpha]^* &\dfn& \dep[x_1,\dots,x_k,x_{\alpha}],\\
\alpha^* &\dfn& x_{\alpha} = x_\top, \\
(\neg\alpha)^* &\dfn& x_{\alpha} = x_\bot, \\
(\varphi \wedge \psi)^*   &\dfn& (\varphi^* \wedge \psi^*),\\
(\varphi \vee \psi)^*   &\dfn& (\varphi^* \vee \psi^*),\\
(\exists \alpha \varphi)^* &\dfn& \exists x_{\alpha}\varphi^*,\\
(\exists x \varphi)^*   &\dfn& \exists x  \varphi^*, \\
(\forall  x \varphi)^*   &\dfn& \forall x  \varphi^*.
\end{eqnarray*}
Clearly, if $\varphi$ is $\bd$-sentence then $\varphi^\mi{+}$ is a $\dl$-sentence. Furthermore, it is easy to see that (\ref{bdtod}) holds for every $\bd$-sentence $\varphi$ and every model $\mA$ of cardinality at least two.
\end{proof}
%
%
\section[Dependence normal form]{Dependence normal form}\label{normalform}
%
In this section we define a normal form for bounded Boolean dependence logic and show that for each \bbd-formula there exists an equivalent $\bbd$-formula in this normal form. We use this normal form in Section \ref{equiv} to establish a translation from $\bbd$ to $\fopocp$. As a byproduct we also obtain translations from $\rbd$ into $\pocfo$ and from $\abd$ into $\pocqf$.

We start by introducing a normal form that does not allow any reuse of variables.
\begin{definition}
A formula $\varphi\in \bd$ is in \emph{variable normal form} if no variable in $\fr
(\varphi)$ is quantified in $\varphi$, and if each variable quantified in $\varphi$ is quantified exactly once.
\end{definition}
Note that, if a $\bd$-formula $\varphi$ is in variable normal form and $\psi$ is a subformula of $\varphi$ that has at least one quantifier in it then $\varphi$ has exactly one occurrence of the subformula $\psi$.
\begin{lemma}\label{varnormalform}
For every $\bd$ ($\bbd$, $\rbd$, $\abd$, respectively) formula there exists an equivalent $\bd$ ($\bbd$, $\rbd$, $\abd$, respectively) formula in variable normal form.
\end{lemma}
\begin{proof}
Follows from Proposition \ref{prop:teamrestriction} and Lemmas \ref{substitution_bd} and \ref{change_variables}.
\end{proof}
The following normal form can be seen as a kind of a local prenex normal form for $\bbd$. The idea is that each universal first-order and existential Boolean quantifier is pulled toward a preceding existential first-order quantifier and then using Boolean dependence atoms each universal first-order quantifier is pulled past the preceding Boolean quantifiers.
\begin{definition}
A sentence $\varphi\in\bd$ is in \emph{$Q$-normal form} if $\varphi$ is in variable normal form and there exists a formula $\vartheta\in\bd$ such that the following holds.
\begin{enumerate}
\item $\varphi =\forall \vec{x}\exists\vec{\alpha}\vartheta$, for some (possibly empty) block of universal quantifiers $\forall \vec{x}$ followed by a (possibly empty) block of existential Boolean quantifiers $\exists\vec{\alpha}$.
\item Each quantifier in $\vartheta$ occurs in some block of quantifiers
$\exists \vec{x}\forall \vec{y}\exists\vec{\alpha}$, where at least $\vec{x}$ is nonempty.
\end{enumerate}
\end{definition}
%
%
The following lemmas are used to prove the $Q$-normal form for \bbd, i.e., Proposition \ref{Qnormalform}.
\begin{lemma}\label{quantifier_swap}
Let $\varphi$ and $\vartheta$ be formulae of Boolean dependence logic such that $x,\alpha\notin \fr(\vartheta)$. The following equivalences hold.
\begin{enumerate}
\item $(\forall x \varphi\lor\vartheta)\equiv \forall x (\varphi\lor\vartheta)$.
\item $(\forall x \varphi\land\vartheta)\equiv \forall x (\varphi\land\vartheta)$.
\item $(\exists \alpha \varphi\lor\vartheta)\equiv \exists \alpha (\varphi\lor\vartheta)$.
\item $(\exists \alpha \varphi\land\vartheta)\equiv \exists\alpha (\varphi\land\vartheta)$.
\end{enumerate}
\end{lemma}
\begin{proof}
Each claim follows straightforwardly from Proposition \ref{prop:teamrestriction}. We prove here claim $1$. Claims 2--4 are completely analogous.

Let $\mA$ be a model and $X$ a team of $A$. The claim follows from the following chain of equivalences.
\begin{align*}
\mA&\models_X (\forall x \varphi\lor\vartheta) \\
&\Leftrightarrow \mA\models_{Y} \forall x \varphi \text{ and } \mA\models_{Z}\vartheta \text{, for some $Y$ and $Z$ such that $Y\cup Z=X$}\\
&\Leftrightarrow \mA\models_{Y(A/x)}\varphi \text{ and } \mA\models_{Z(A/x)}\vartheta \text{, for some $Y$ and $Z$ such that $Y\cup Z=X$}\\
&\Leftrightarrow \mA\models_{Y'}\varphi \text{ and } \mA\models_{Z'}\vartheta \text{, for some $Y'$ and $Z'$ such that $Y'\cup Z'=X(A/x)$}\\
&\Leftrightarrow \mA\models_{X(A/x)} (\varphi\vee\vartheta)\\
&\Leftrightarrow \mA\models_{X} \forall x(\varphi\vee\vartheta).
\end{align*}
The first and the fourth equivalence is due to the semantics of disjunctions. The second equivalence follows from the semantics of universal quantifiers, Proposition \ref{prop:teamrestriction} and the fact that $x\not\in\fr(\vartheta)$. The third equivalence follows from the observation that $Y(A/x)\cup Z(A/x)= X(A/x)$, from Proposition \ref{prop:teamrestriction} and the fact that $x\not\in\fr(\vartheta)$. Finally, the last equivalence is due to the semantics of universal quantifiers.
\end{proof}
\begin{lemma}\label{quantifier_order}
Let $\varphi$ be a \bd-sentence and
$\psi = \exists \alpha \forall \vec{x}\, \exists \vec{\beta}\, \vartheta$
a subformula of $\varphi$. Let $[\psi,\varphi]_n$ denote an occurrence of $\psi$ in $\varphi$ and let $\varphi^*$ denote the formula obtained form $\varphi$ by substituting the occurrence $[\psi,\varphi]_n$ of $\psi$ by
\[
\forall \vec{x}\, \exists \alpha \exists \vec{\beta}\, \big(\dep[{V([\psi,\varphi]_n), \alpha}] \land \vartheta\big).
\]
Then 
$
\varphi \equiv \varphi^*.
$
\end{lemma}
\begin{proof}
Straightforward.
\end{proof}
We are now ready to prove that for every $\bbd$-sentence there exists an equivalent $\bbd$-sentence in $Q$-normal form.
\begin{proposition}\label{Qnormalform}
For every $\bbd$-sentence there exists an equivalent $\bbd$-sentence in $Q$-normal form.
\end{proposition}
\begin{proof}
Let $\varphi\in\bbd$ be a sentence. By Lemma~\ref{varnormalform} we can assume that $\varphi$ is in variable normal form. We will give an algorithm that transforms $\varphi$ into an equivalent $\bbd$-sentence in $Q$-normal form. 

We will first transform $\varphi$ to an equivalent $\bbd$-sentence $\varphi^*$ such that
\begin{equation}\label{eq:transformation}
\varphi^*=\vv{Q \xi} \psi,
\end{equation}
where $\vv{Q \xi}$ is a (possibly empty) vector of universal first-order and existential Boolean quantifiers. Furthermore, in $\psi$ every universal first-order or existential Boolean quantifier $Q \chi$ occurs in a subformula $\vartheta$ of $\psi$ such that
\[
\vartheta=Q'\eta Q\chi\gamma,
\]
where $Q' \eta$ is a quantifier and $\gamma$ is a $\bbd$ formula. In order to obtain $\varphi^*$ from $\varphi$, we use the equivalences from Lemma \ref{quantifier_swap} repetitively substituting subformulae with equivalent subformulae. More precisely, there exists a natural number $n\in\mathbb{N}$ and a tuple $(\varphi_i)_{i\leq n}$ of $\bbd$-sentences such that $\varphi_0=\varphi$ and $\varphi_n=\varphi^*$. Furthermore,
\begin{enumerate}
\item\label{item:substitution}
for each $i<n$ there exist subformulae $\vartheta$, $\psi_1$ and $\psi_2$ of $\varphi_i$ such that
\[
\vartheta=(Q\chi\psi_1\otimes\psi_2) \text{ (or $\vartheta=(\psi_1\otimes Q\chi\psi_2)$ )},
\]
where $Q\chi\in\{\forall x,\exists \alpha\}$ and $\otimes\in\{\vee,\wedge\}$, and
$\varphi_{i+1}$ is obtained from $\varphi_{i}$ by substituting $\vartheta$ by $Q\chi(\psi_1\otimes\psi_2)$.
\end{enumerate}
It is easy to see, that for each $\bbd$-sentence the substitution procedure described in \ref{item:substitution} terminates, i.e., there exists some $n\in \mathbb{N}$ such that there are no subformulae $\vartheta$, $\psi_1$ and $\psi_2$ of $\varphi_n$ such that the substitution described in \ref{item:substitution} can be carried out. Clearly the sentence $\varphi_n$ is then in the form described in (\ref{eq:transformation}).
By induction it is easy to show, that since $\varphi_0$ is in variable normal form it follows that $\varphi_i$ is in variable normal form, for all $i \geq 0$. Hence, the assumptions on free variables needed for Lemma~\ref{quantifier_swap} hold for each $\varphi_i$.
By Lemmas \ref{quantifier_swap} and \ref{substitution_bd}, we conclude that, for each $i<n$, the sentences $\varphi_i$ and $\varphi_{i+1}$ are equivalent. Hence the sentences $\varphi_0$ and $\varphi_n$ are equivalent.

We still need to transform the sentence $\varphi^*$ into an equivalent sentence $\varphi'$ in $Q$-normal form. In order to obtain $\varphi'$ from $\varphi^*$ we use the equivalence from Lemma \ref{quantifier_order} repetitively.
More precisely, there exists a natural number $m\in\mathbb{N}$ and a tuple $(\varphi^*_i)_{i\leq m}$ of $\bbd$-sentences such that $\varphi^*_0=\varphi^*$ and $\varphi^*_m=\varphi'$. Furthermore,
\begin{enumerate}
\setcounter{enumi}{1}
\item\label{item:substitution2} for each $i<m$ there exists subformulae $\vartheta$ and $\psi$ of $\varphi^*_i$, and a quantifier $\exists \alpha$ such that
\[
\vartheta=\exists \alpha  \forall\vec{x} \exists\vec{\beta} \psi,
\]
where $\forall\vec{x}$ is a nonempty vector of universal first-order quantifiers and $\psi$ does not start with a Boolean existential or universal first-order quantifier, and
$\varphi^*_{i+1}$ is obtained from $\varphi^*_{i}$ by substituting $\vartheta$ by
\[
\forall\vec{x}\, \exists \alpha\, \exists\vec{\beta}\,\big(\dep[{V([\vartheta,\varphi^*_i]),\alpha}]\wedge\psi\big).
\]
\end{enumerate}
By Lemmas \ref{quantifier_order} and \ref{substitution_bd}, we conclude that, for each $i$, the sentences $\varphi^*_i$ and $\varphi^*_{i+1}$ are equivalent. It is easy to see that, for each $\bbd$-sentence the substitution procedure described above terminates, i.e., there exists some $m\in\mathbb{N}$ such that there are no subformulae of $\varphi^*_m$ that can be substituted as described in \ref{item:substitution2}. Now clearly $\varphi^*_m$ is in $Q$-normal form.
\end{proof}
We are finally ready to define dependence normal form. The idea here is that a $\bbd$-sentence in $Q$-normal form is in dependence normal form if there is one-to-one correspondence between Boolean existential quantifiers and Boolean dependence atoms such that each quantifier $\exists \alpha$ is immediately followed by the corresponding dependence atom $\dep[\vec{x},\alpha]$, and conversely each dependence atom is directly preceded by the corresponding Boolean quantifier.
\begin{definition}\label{depnor}
A sentence $\varphi\in\bd$ is in \emph{dependence normal form} if
\begin{enumerate}
\item $\varphi$ is in $Q$-normal form,
\item for every Boolean variable $\alpha$ it holds that if $[\dep[\vec{x},\alpha],\varphi]_t$ and $[\dep[\vec{y},\alpha],\varphi]_l$ are occurrences in  $\varphi$ then $[\dep[\vec{x},\alpha],\varphi]_t=[\dep[\vec{y},\alpha],\varphi]_l$,
%
%
\item
for every maximal nonempty block of Boolean existential quantifiers $\exists\vec{\alpha}$ in $\varphi$ there exists a subformula 
\[
\exists\vec{\alpha}\Big(\big(\bigwedge_{1\leq i\leq n}\dep[\vec{x}_i,\alpha_i]\big)\wedge \psi\Big)
\]
of $\varphi$ such that the Boolean variables $\alpha_i$, $1\leq i\leq n$, are exactly the variables quantified in $\exists\vec{\alpha}$.
\end{enumerate}
\end{definition}
To simplify the proof of Proposition \ref{dep_nf}, we introduce the concepts of evaluation and satisfying evaluation.
\begin{definition}\label{evaluation}
Let $\mA$ be a model, $X$ a team of $\mA$ and $\varphi$ a $\bd$-formula. Define that 
\begin{align*}
\mi{SubOc}(\varphi)&\dfn \{[\psi,\varphi]_t\mid \text{$[\psi,\varphi]_t$ is an occurrence of $\psi$ in $\varphi$} \}\\
\mi{Teams}(\mA)&\dfn \{Y\mid \text{$Y$ is a team of $\mA$}\}.
\end{align*}
We say that a function $e:\mi{SubOc}(\varphi)\to \mi{Teams}(\mA)$ is an \emph{evaluation} of $\varphi$ on the model $\mA$ and team $X$ if the following recursive conditions hold.
\begin{enumerate}
\item $e([\varphi, \varphi]) = X$.
\item If $e([\psi\vee\vartheta, \varphi]_t)= Y$, and $[\psi, \varphi]_n$ and $[\vartheta, \varphi]_m$ are the occurrences of $\psi$ and $\vartheta$ in $[\psi\vee\vartheta, \varphi]_t$, then there exists $Y_0$ and $Y_1$ such that $Y_0\cup Y_1=Y$ and $e([\psi, \varphi]_n)= Y_0$ and $e([\vartheta, \varphi]_m)= Y_1$.
\item If $e([\psi\wedge\vartheta, \varphi]_t)= Y$, and $[\psi, \varphi]_n$ and $[\vartheta, \varphi]_m$ are the occurrences of $\psi$ and $\vartheta$ in $[\psi\wedge\vartheta, \varphi]_t$, then $e([\psi, \varphi]_n)= Y$ and $e([\vartheta, \varphi]_m)= Y$.
\item If $e([\exists x\psi, \varphi]_t)= Y$ and $[\psi,\varphi]_{n}$ is the occurrence of $\psi$ in $[\exists x\psi, \varphi]_t$ then $e([\psi, \varphi]_n)= Y(F/x)$ for some function $F:Y\to A$.
\item If $e([\exists \alpha\psi, \varphi]_t)= Y$ and $[\psi,\varphi]_{n}$ is the occurrence of $\psi$ in $[\exists \alpha\psi, \varphi]_t$ then $e([\psi, \varphi]_n)= Y(F/\alpha)$ for some function $F:Y\to \{\bot,\top\}$.
\item If $e([\forall x\psi, \varphi]_t)= Y$ and $[\psi,\varphi]_{n}$ is the occurrence of $\psi$ in $[\forall x\psi, \varphi]_t$ then $e([\psi, \varphi]_n)= Y(A/x)$.
\end{enumerate}
We say that the evaluation $e$ is an \emph{successful evaluation} if for each occurrence $[\psi,\varphi]_t$ such that $\psi$ is a Boolean dependence atom or a literal
\[
\mA \models_{e([\psi,\vartheta]_t)} \psi.
\]
\end{definition}
The following results follow directly from the semantics of Boolean dependence logic and the definition of successful evaluation.
\begin{proposition}\label{sucevap}
Let $\mA$ be a model, $X$ a team of $\mA$, $\varphi$ a $\bd$-formula and $e$ a successful evaluation of $\varphi$ on the model $\mA$ and team $X$. For every occurrence $[\psi, \varphi]_t\in \mi{SubOc}(\varphi)$
\[
\mA\models_{e([\psi, \varphi]_t)} \psi.
\]   
\end{proposition}

\begin{theorem}\label{suceva}
Let $\mA$ be a model, $X$ a team of $\mA$ and $\varphi$ a $\bd$-formula. The following are equivalent:
\begin{enumerate}
\item $\mA\models_X \varphi$.
\item There exists a successful evaluation of $\varphi$ on the model $\mA$ and team $X$.
\end{enumerate}
\end{theorem}
Hence the concepts of satisfaction and successful evaluation coincide. We are now ready to prove that every $\bbd$-sentence there exists an equivalent $\bbd$-sentence in dependence normal form.
The proof is quite long and technical.
\begin{proposition}\label{dep_nf}
For every $\bbd$-sentence there exists an equivalent $\bbd$-sentence in dependence normal form.
\end{proposition}
\begin{proof}
Let $\varphi\in\bbd$ be a sentence. By Proposition \ref{Qnormalform}, we may assume that $\varphi$ is in $Q$-normal form. We will give an algorithm that transforms $\varphi$ to an equivalent $\bbd$-sentence $\varphi^+$ in dependence normal form. We show that there exists
a natural number $n\in\mathbb{N}$ and
a tuple $(\varphi_i)_{i\leq n}$ of equivalent $\bbd$-sentences in $Q$-normal form such that $\varphi_0=\varphi$ and $\varphi_n=\varphi^+$.
The sentence $\varphi_{i+1}$ is obtained from $\varphi_i$ by the procedure described below.

Assume that $\varphi_i$ is not in dependence normal form. Assume first that this is due to the fact that there exists an occurrence
\[
[\dep[\vec{x},\alpha],\varphi_i]_t
\]
of some dependence atom $\dep[\vec{x},\alpha]$ in $\varphi_i$ that violates the conditions of Definition \ref{depnor}, i.e., there exists some other occurrence of a dependence atom with a consequent $\alpha$ or $[\dep[\vec{x},\alpha],\varphi_i]_t$ is not a conjunct in a conjunction of dependence atoms immediately following a block of existentially quantified Boolean variables in which $\exists \alpha$ occurs. 
Hence, there exists a formula $\psi$ and maximal quantifier blocks $\exists\vec{y}$, $\forall \vec{z}$ and $\exists\vec{\beta}$ such that $\exists\vec{\beta}$ is nonempty and
\[
\vartheta \dfn \exists\vec{y}\,\forall \vec{z}\,\exists\vec{\beta}\, \psi
\]
is a subformula of $\varphi_i$, where the occurrence $[\dep[\vec{x},\alpha],\varphi_i]_t$ of $\dep[\vec{x},\alpha]$ is a subformula of $\psi$ and is not bound by any quantifier in $\psi$.
%
%
Let $U$ denote the set of variables that are in both $\vec{x}$ and
$\vec{z}$.
By $\vec{u}$ we denote the canonical ordering of the variables in $U$.
Since $\varphi_i$ is a $\bbd$-sentence, the variables in $\vec{x}$ are exactly those that are in $V([\forall \vec{z}\,\exists\vec{\beta}\, \psi,\varphi_i])\cup U$. Hence the formulae 
\[
\dep[{V([\forall \vec{z}\,\exists\vec{\beta}\, \psi, \varphi_i])\cup U,\alpha}] \quad\text{ and }\quad \dep[\vec{x},\alpha]
\]
are equivalent. Therefore, due to Lemma \ref{substitution_bd}, we may assume that $\dep[\vec{x},\alpha]$ is $\dep[{V([\forall \vec{z}\,\exists\vec{\beta}\, \psi, \varphi_i])\cup U,\alpha}]$.

Let $\vec{w}$ be a tuple of fresh distinct first-order variables of the same length as $\vec{u}$,
$W$ the set of variables in $\vec{w}$ and
$\beta'$ a fresh Boolean variable. Define then that
\[
\vartheta':=\exists \vec{y}\,\forall \vec{z}\,\forall \vec{w}\,\exists\vec{\beta}\,\exists\beta'\,\big(\dep[{V([\forall \vec{z}\,\exists\vec{\beta}\, \psi, \varphi_i])\cup W,\beta'}]\wedge\psi'\big),
\]
where
$\psi'$ is obtained from $\psi$
by substituting the occurrence $[\dep[\vec{x},\alpha],\varphi_i]_t$ of $\dep[\vec{x},\alpha]$ by $\vec{u}=\vec{w}\rightarrow \alpha=\beta'$.
%
%
%
We will show that the formulae $\vartheta$ and $\vartheta'$ are equivalent. First observe that, since the variables $\vec{w},\beta'$ do not occur in $\psi$, it is easy to conclude, by Proposition \ref{prop:teamrestriction}, that $\vartheta$ is equivalent to the formula
\[
\gamma:=\exists \vec{y}\,\forall \vec{z}\,\forall \vec{w}\,\exists\vec{\beta}\,\exists\beta'\,\big(\dep[{V([\forall \vec{z}\,\exists\vec{\beta}\, \psi, \varphi_i])\cup W,\beta'}]\wedge\psi\big).
\]
We still need to show that $\gamma$ is equivalent to $\vartheta'$. 
Note that $\vartheta'$ can be obtained from $\gamma$ by substituting one occurrence of $\dep[\vec{x},\alpha]$ by $\vec{u}=\vec{w}\rightarrow \alpha=\beta'$.
Notice also that
\begin{align}
\mA&\models_{Z} \dep[{V([\forall \vec{z}\,\exists\vec{\beta}\, \psi, \varphi_i])\cup W,\beta'}]\wedge\psi, \label{middle}\\ 
\intertext{if and only if}
\mA&\models_{Z'} \dep[{V([\forall \vec{z}\,\exists\vec{\beta}\, \psi, \varphi_i])\cup W,\beta'}]\wedge\psi \label{end},
\end{align}
where $\mA$ is a model, and $Z$ and $Z'$ are teams of $\mA$ such that
\begin{align*}
&Z\upharpoonright \fr(\psi)=Z'\upharpoonright \fr(\psi),\\
&\mA\models_{Z} \dep[{V([\forall \vec{z}\,\exists\vec{\beta}\, \psi, \varphi_i])\cup W,\beta'}], \text{ and}\\
&\mA\models_{Z'} \dep[{V([\forall \vec{z}\,\exists\vec{\beta}\, \psi, \varphi_i])\cup W,\beta'}].
\end{align*}
%
%
%
%
%
%
Therefore and since $\lvert \vec{w}\rvert=\lvert\vec{u}\rvert$, we can encode a partial function related to the occurrence $[\dep[\vec{x},\alpha],\varphi_i]_t$ of the dependence atom $\dep[\vec{x},\alpha]$ using $\vec{w}$ and $\beta'$.

Assume first that $\mA\models_X \gamma$. Hence
\[
\mA\models_Y \dep[{V([\forall \vec{z}\,\exists\vec{\beta}\, \psi, \varphi_i])\cup W,\beta'}]\wedge\psi,
\]
for some team $Y$ that can be obtained from $X$ by evaluating the quantifier prefix of $\gamma$.
Therefore, by Theorem \ref{suceva}, there exists some successful evaluation $e$ of
\[
\dep[{V([\forall \vec{z}\,\exists\vec{\beta}\, \psi, \varphi_i])\cup W,\beta'}]\wedge\psi
\]
on the model $\mA$ and team $Y$. Hence, by Proposition \ref{sucevap},
\[
\mA\models_{e\big([\dep[\vec{x},\alpha],\varphi_i]_t\big)}\dep[\vec{x},\alpha],
\]
i.e.,
\[
\mA\models_{e\big([\dep[\vec{x},\alpha],\varphi_i]_t\big)} \dep[{V([\forall \vec{z}\,\exists\vec{\beta}\, \psi, \varphi_i])\cup U,\alpha}].
\]
Therefore, there exists a partial function
\[
f_e:A^{\lvert V([\forall \vec{z}\,\exists\vec{\beta}\, \psi, \varphi_i])\cup U \rvert}\rightarrow \{\bot,\top\}
\]
that maps the values of the variables of $V([\forall \vec{z}\,\exists\vec{\beta}\, \psi, \varphi_i])\cup U$ to the value of $\alpha$ in the team $e\big([\dep[\vec{x},\alpha],\varphi_i]_t)$. Remember that $\lvert U\rvert = \lvert W\rvert$. Let
\[
g_e:A^{\lvert V([\forall \vec{z}\,\exists\vec{\beta}\, \psi, \varphi_i])\cup W \rvert}\rightarrow\{\bot,\top\}
\]
be a function such that $f_e\subseteq g_e$, and let $Y'$ be the variant of $Y$ for which the values for $\beta'$ have been picked by using the function $g_e$, i.e., such that $Y\upharpoonright\fr(\psi)=Y'\upharpoonright\fr(\psi)$ and such that, for every $s\in Y'$,
\[
s(\beta')=g_e\big(s(V([\forall \vec{z}\,\exists\vec{\beta}\, \psi, \varphi_i])\cup W \rvert)\big).
\]
Hence the equivalence of \eqref{middle} and \eqref{end} can be applied here.
Therefore and since $\mA\models_{Y} \dep[{V([\forall \vec{z}\,\exists\vec{\beta}\, \psi, \varphi_i])\cup W,\beta'}]$, we conclude that
\[
\mA\models_{Y'} \dep[{V([\forall \vec{z}\,\exists\vec{\beta}\, \psi, \varphi_i])\cup W,\beta'}]\wedge\psi.
\]
Now since $Y\upharpoonright\fr(\psi)=Y'\upharpoonright\fr(\psi)$, it follows from Lemma \ref{prop:teamrestriction} and Theorem \ref{suceva} that there exists some successful evaluation $e'$ of
\[
\dep[{V([\forall \vec{z}\,\exists\vec{\beta}\, \psi, \varphi_i])\cup W,\beta'}]\wedge\psi
\]
on the model $\mA$ and team $Y'$ such that
\[
e'\big([\dep[\vec{x},\alpha],\varphi]_t\big)\upharpoonright\fr(\psi)\,=\, e\big([\dep[\vec{x},\alpha],\varphi]_t\big)\upharpoonright\fr(\psi).
\]
By Proposition \ref{sucevap},
\[
\mA\models_{e'\big([\dep[\vec{x},\alpha],\varphi]_t\big)} \, \dep[\vec{x},\alpha],
\]
i.e.,
\[
\mA\models_{e'\big([\dep[\vec{x},\alpha],\varphi]_t\big)} \, \dep[{V([\forall \vec{z}\,\exists\vec{\beta}\, \psi, \varphi_i])\cup U,\alpha}].
\]
Therefore and since the values for $\beta'$ in $Y'$ were picked by applying the expansion $g_e$ of $f_e$ to the values of $V([\forall \vec{z}\,\exists\vec{\beta}\, \psi, \varphi_i])$ and $\vec{w}$, it follows that the values of $\beta'$ in $e'\big([\dep[\vec{x},\alpha],\varphi]_t\big)$ can be obtained by applying the expansion $g_e$ of $f_e$ to the values of $V([\forall \vec{z}\,\exists\vec{\beta}\, \psi, \varphi_i])$ and $\vec{w}$. Thus
\[
\mA\models_{e'\big([\dep[\vec{x},\alpha],\varphi]_t\big)} \vec{w}=\vec{u}\rightarrow \alpha=\beta'.
\]
Hence, we conclude that there exists a successful evaluation of
\[
\dep[{V([\forall \vec{z}\,\exists\vec{\beta}\, \psi, \varphi_i])\cup W,\beta'}]\wedge\psi'
\]
on the model $\mA$ and team $Y'$. Therefore, by Theorem \ref{suceva},
\[
\mA\models_{Y'} \dep[{V([\forall \vec{z}\,\exists\vec{\beta}\, \psi, \varphi_i])\cup W,\beta'}]\wedge\psi'.
\]
Clearly $Y'$ can be obtained form $X$ by evaluating the quantifier prefix of $\vartheta'$. Hence
\[
\mA\models_X \vartheta'.
\]

Assume then that $\mA\models_X \vartheta'$. Hence
\[
\mA\models_Y \dep[{V([\forall \vec{z}\,\exists\vec{\beta}\, \psi, \varphi_i])\cup W,\beta'}]\wedge\psi',
\]
for some team $Y$ that can be obtained from $X$ by evaluating the quantifier prefix of $\vartheta'$. Thus, there exists some successful evaluation $h$ of $\psi'$ on the model $\mA$ and team $Y$. Hence
\[
\mA\models_{h([\vec{w}=\vec{u}\rightarrow \alpha=\beta',\vartheta'])} \vec{w}=\vec{u}\rightarrow \alpha=\beta'.
\]
Since the variables in $\vec{w}$ and $\beta'$ do not occur in other subformulae of $\psi'$ other than $\vec{w}=\vec{u}\rightarrow \alpha=\beta'$, we may assume, by Proposition \ref{prop:teamrestriction}, that for each assignment $s\in h([\vec{w}=\vec{u}\rightarrow \alpha=\beta',\vartheta'])$ and $\vec{a}\in A^{\lvert\vec{w} \rvert}$ the modified assignment $s'\in Y$ of $s$ that maps $\vec{w}$ to $\vec{a}$ and $\beta'$ to $\bot$ or $\top$ is also in $h([\vec{w}=\vec{u}\rightarrow \alpha=\beta',\vartheta'])$. Now since
\begin{align*}
\mA&\models_Y\dep[{V([\forall \vec{z}\,\exists\vec{\beta}\, \psi, \varphi_i])\cup W,\beta'}],\\
\mA&\models_{h([\vec{w}=\vec{u}\rightarrow \alpha=\beta',\vartheta'])} \vec{w}=\vec{u}\rightarrow \alpha=\beta',
\end{align*}
and $h([\vec{w}=\vec{u}\rightarrow \alpha=\beta',\vartheta'])\subseteq Y$, we conclude that
\[
\mA\models_{h([\vec{w}=\vec{u}\rightarrow \alpha=\beta',\vartheta'])}\dep[{V([\forall \vec{z}\,\exists\vec{\beta}\, \psi, \varphi_i])\cup U,\alpha}].
\]
Remember that $\psi$ can be obtained from $\psi'$ by substituting $\vec{w}=\vec{u}\rightarrow \alpha=\beta'$ with $\dep[{V([\forall \vec{z}\,\exists\vec{\beta}\, \psi, \varphi_i])\cup U,\alpha}]$. Thus, $h$ can be modified into a successful evaluation of
\(
\psi
\)
on the model $\mA$ and team $Y$. Therefore, by Theorem \ref{suceva},
\[
\mA\models_Y \psi.
\]
Since $\mA\models_Y \dep[{V([\forall \vec{z}\,\exists\vec{\beta}\, \psi, \varphi_i])\cup W,\beta'}]$ and $Y$ can clearly be obtained from $X$ by evaluating the quantifier prefix of $\gamma$, we conclude that $\mA\models_X \gamma$. Thus we have shown that $\gamma$ and $\vartheta'$ are equivalent. Since $\gamma$ and $\vartheta$ are equivalent, we can finally conclude that $\vartheta$ and $\vartheta'$ are equivalent.

Let $\varphi_{i+1}$ be the sentence obtained from $\varphi_i$ by substituting $\vartheta$ with $\vartheta'$. Since $\vartheta$ and $\vartheta'$ are equivalent, it follows from Lemma $\ref{substitution_bd}$ that $\varphi_i$ and $\varphi_{i+1}$ are equivalent. Notice that, if $\varphi_i$ is in $Q$-normal form, then $\varphi_{i+1}$ is also in $Q$-normal form.
Furthermore, in $\varphi_{i+1}$ there is strictly less\footnote{To be precise, to achieve this we may need to reorder the conjunction in which $[\dep[\vec{x},\alpha],\varphi_i]_t$ was a conjunct of.} occurrences of dependence atoms that violate the condition of Definition \ref{depnor} than in $\varphi_{i}$.
Hence
for large enough
$k$, the formula $\varphi_k$ does not have any dependence atoms that violate the conditions of Definition \ref{depnor}. Hence, if $\varphi_k$ is not in dependence normal form there exists a subformula
\[
\exists\beta\,\exists\vec{\alpha}\,\psi
\]
of $\varphi_k$ such that $\vv{\alpha}$ is maximal and such that $\dep[\vec{x},\beta]$ is not a subformula of $\psi$ for any $\vec{x}$. Let $\varphi_{k+1}$ denote the formula obtained from $\varphi_k$ by substituting
\[
\exists\beta\,\exists\vec{\alpha}\,\psi \quad\text{ by }\quad \exists\beta\,\exists\vec{\alpha}\, \big(\dep[{V([\exists\beta\,\exists\vec{\alpha}\,\psi, \varphi_k]),\beta}]\wedge\psi\big).
\]
Clearly $\varphi_k$ and $\varphi_{k+1}$ are equivalent. It is easy to see that the procedure described here terminates and finally produces an equivalent sentence in dependence normal form.
\end{proof}

%
\section{Fragments of $\fopoc$ and $\bd$ coincide}\label{equiv}
In this section we use the normal form for bounded Boolean dependence logic from Section \ref{normalform} to establish that
\[
\bbd\equiv\fopocp,\, \rbd\equiv\pocfo\, \text{ and }\, \abd\equiv\pocqf.
\]
In addition, we show that $\bd\equiv \D$.

\begin{definition}
Let $\varphi\in\bbd$ be a sentence in dependence normal form. We say that 
a subformula $\psi$ of $\varphi$ is \emph{dependence maximal} (with respect to $\varphi$) if either
$\psi$ does not contain any dependence atoms, or
\[
	\psi = \forall \vec{x}\,\exists\vec{\alpha}\,\vartheta,
\]
where $\vec{\alpha}$ is nonempty and neither $\forall y \psi$ nor $\exists \beta\exists\vec{\alpha}\,\vartheta$ is a subformula of $\varphi$, for any $y$ or $\beta$.
\end{definition}

\begin{theorem}\label{bbd-fopocp}
$\bbd \equiv \fopocp$.
\end{theorem}
\begin{proof}
We will first prove that $\bbd\leq\fopocp$. Let $\varphi$ be a $\bbd$-sentence. By Proposition~\ref{dep_nf} we may assume that $\varphi$ is in dependence normal form. We will translate $\varphi$ into an equivalent $\fopocp$ sentence $\varphi^*$ by substituting each maximal block $\forall\vec{x}\,\exists\vec{\alpha}$ of quantifiers along with the corresponding dependence atoms in $\varphi$ by a partially-ordered connective.

More precisely, we define a translation $\psi\mapsto\psi^*$ for all subformulae $\psi$ of $\varphi$ that
are dependence maximal or can be obtained from dependence maximal subformulae of $\varphi$ by first-order operations, i.e., by taking conjunctions, disjunctions and first-order quantifications. Note that, every $\bbd$-sentence that is in dependence normal form can be build from its dependence maximal subformulae by using only first-order operations. 
The translation is defined recursively as follows: 
\begin{enumerate}
\item[(i)] If $\psi$ is a formula without dependence atoms then
$\psi^*:=\psi$.
\item[(ii)] If 
\[
	\psi=\forall \vec{x}\, \exists \vec{\alpha}\, 
	\Big(\big(\bigwedge_{1\leq i\leq m} \dep[{V([\psi,\varphi])\cup
	\{\vec{x}_i\},\alpha_i}]\big)\wedge  \vartheta \Big)
\]
is dependence maximal and $\vec{\alpha}=(\alpha_1,\dots,\alpha_m)$, we define that
%
\begin{equation}\label{eq:partial}
\psi^*\dfn N_\pi\,\vec{x}_0\alpha_0\vec{x}_1\alpha_1
\ldots\vec{x}_m\alpha_m\, \vartheta^*,
\end{equation}
where $\vec{x}_0$ are exactly those variables in $\vec{x}$ that are not in any $\vec{x}_i$, 
$1\leq i\leq m$, and $\alpha_0$ is a fresh Boolean variable not occurring in $ \vartheta^*$ nor $\varphi$. The pattern $\pi$ of the connective is obtained canonically form the identities between the variables in the tuples $\vec{x}_i$, $i\leq m$.
\item[(iii)] If $\psi= (\vartheta\land\eta)$, we define that 
$\psi^*\dfn  (\vartheta^*\land\eta^*)$.
\item[(iv)] If $\psi= (\vartheta\lor\eta)$, we define that 
$\psi^*\dfn  (\vartheta^*\lor\eta^*)$.
\item[(v)] If $\psi= \exists x\,\eta$, we define that $\psi^*\dfn \exists x \,\eta^*$.
\item[(vi)] If $\psi= \forall x\,\eta$ and $\psi$ is not dependence maximal, we define that $\psi^*\dfn \forall x\, \eta^*$.
\end{enumerate}
Note that since $\varphi$ is in dependence normal form, $\varphi^*$ is defined, and 
clearly $\varphi^*\in\fopocp$.
Thus, it suffices to show that for every formula $\psi$ that can be obtained from dependence maximal subformulae of $\varphi$ by using conjunctions, disjunctions and first-order quantifications, for every model $\mA$ and for every team $X$ of $\mA$ such that $\fr(\psi)\subseteq \dom(X)$
\[
	\mA\models_X\psi\quad\Leftrightarrow\quad\mA,s\models\psi^* \text{ for all $s\in X$.}
\]


The proof is done by induction on the definition of the translation.

(i) If $\psi$ is without dependence atoms,
the claim holds by Proposition \ref{FO extensionbd}.

(ii) Assume that
\[
	\psi= \forall \vec{x}\, \exists \vec{\alpha}\, 
	\Big(\big(\bigwedge_{1\leq i\leq m} \dep[{V([\psi,\varphi])\cup
	\{\vec{x}_i\},\alpha_i}]\big)\wedge  \vartheta \Big)
\]
is dependence maximal and that $\vec{\alpha}=(\alpha_1,\dots,\alpha_m)$.
Then
\[
\psi^*= N_\pi\,\vec{x}_0\alpha_0\vec{x}_1\alpha_1
\ldots\vec{x}_m\alpha_m\, \vartheta^*,
\]
where $\pi$ is the pattern determined by the tuple $\vec{x}_0,\vec{x}_1,\ldots,\vec{x}_m$. 
We will show that
\[
\mA\models_X\psi\quad\Leftrightarrow\quad\mA,s\models N_\pi\,\vec{x}_0\alpha_0\vec{x}_1\alpha_1
\ldots\vec{x}_m\alpha_m\, \vartheta^* \text{ for all $s\in X$,}
\]
for every model $\mA$ and every team $X$ of $\mA$ such that $\fr(\psi)\subseteq \dom(X)$.

Let $\mA$ be a model and $X$ a team of $\mA$ such that $\fr(\psi)\subseteq \dom(X)$.
Furthermore, define that $n\dfn|\vec{x}|$ and that $n_i\dfn|\vec{x}_i|$, for each $i\le m$. 
Now, by the semantics of the quantifiers,
\[
\mA\models_X\psi
\]
if and only if for each $i$, $1\leq i\leq m$, there exists a function 
\[
F_i:X(A^n/\vec{x},F_1/\alpha_1,\dots,F_{i-1}/\alpha_{i-1})\to \{\bot,\top\}
\]
such that
\begin{equation}\label{trans-eq}
	\mA\models_Y \Big(\bigwedge_{1\leq i\leq m} 
	\dep[{V([\psi,\varphi])\cup\{\vec{x}_i\},\alpha_i}]\Big)\wedge  \vartheta,
\end{equation}
where $Y=X(A^n/\vec{x},\vec{F}/\vec{\alpha})$.
For each $i$, $1\leq i\leq m$, we define
\[
G_i:X(A^n/\vec{x})\to \{\bot,\top\}
\]
to be the unique function obtained from $F_i$ such that, for every $s\in \dom(F_i)$,
\[
G_i\big(s\upharpoonright \dom(G_i)\big) = F_i(s).
\]
Assume first that (\ref{trans-eq}) holds. For each $s\in X$ and $1\leq i \leq m$, let
\[
f^s_i: A^{n_i}\to\{\bot,\top\}
\]
denote the function such that 
\[
f^s_i(\vec{a}_i)=G_i\big(s(\vec{a}/\vec{x})\big),
\]
where $\vec{a}_i$ is the restriction of $\vec{a}$ to the variables $\vec{x}_i$.
Note that $f^s_i$ is well-defined, since by the first conjunct of (\ref{trans-eq}), the function $F_i$ and hence the function $G_i$ is 
$V([\psi,\varphi])\cup\{\vec{x}_i\}$-determined.
Furthermore, we define that
\[
f^s_0(\vec{a}_0):=\top,
\]
for every $s\in X$ and $\vec{a}_0\in A^{n_0}$. 
If $\vec{a}_i\in A^{n_i}$, $i\leq m$, are tuples such that
$\vec{a}_0\,\vec{a}_1\ldots\vec{a}_m$
is of pattern $\pi$ then clearly, for every $s\in X$, the modified assignment
\[
	s':=s(\vec{a}_0/\vec{x}_0,\ldots,\vec{a}_m/\vec{x}_m,
	f^s_1(\vec{a}_1)/\alpha_1,\ldots, f^s_m(\vec{a}_m)/\alpha_m)
\]
is in $Y$. Now since $\mA\models_Y \vartheta$, by induction hypothesis, we have that
$\mA,s'\models \vartheta^*$, for every $s\in X$. Therefore, since $\alpha_0$ does not occur in $\vartheta^*$, we have that
\[
\mA,s'(\top/\alpha_0)\models  \vartheta^*,
\]
for every $s\in X$.
Hence, the functions
$f^s_i$, $i\le m$, are as required in the truth condition of 
$N_\pi$, and we conclude that
\[
\mA,s\models N_\pi\,\vec{x}_0\alpha_0\vec{x}_1\alpha_1
\ldots\vec{x}_m\alpha_m\, \vartheta^*,
\]
for every $s\in X$.

Assume then that
\[
\mA,s\models N_\pi\,\vec{x}_0\alpha_0\vec{x}_1\alpha_1
\ldots\vec{x}_m\alpha_m\, \vartheta^*
\]
holds for every $s\in X$.
Hence, for every $s\in X$ and $i\leq m$, there exists a function
\[
f^s_i:A^{n_i}\to\{\bot,\top\}, 
\]
such that if $\vec{a}_0,\dots,\vec{a}_m$ is of pattern $\pi$ then
\[
\mA,s(\vec{a}_0/\vec{x}_0,\ldots,\vec{a}_m/\vec{x}_m, f^s_0(\vec{a}_0)/\alpha_0,\ldots, f^s_m(\vec{a}_m)/\alpha_m)\models \vartheta^*.
\]
Since $\alpha_0$ does not occur in $\vartheta^*$ we conclude that
\[
\mA,s(\vec{a}_0/\vec{x}_0,\ldots,\vec{a}_m/\vec{x}_m, f^s_1(\vec{a}_1)/\alpha_1,\ldots, f^s_m(\vec{a}_m)/\alpha_m)\models \vartheta^*.
\]
%
%
Now, for each $i$, $1\leq i\leq m$, define the function 
\[
F_i:X(A^n/\vec{x},F_1/\alpha_1,\dots,F_{i-1}/\alpha_{i-1})\to \{\bot,\top\}
\]
by setting that
\[
F_i\Big(s\big(\vec{a}/\vec{x}, f_1^s(\vec{a}_1)/\alpha_1, \dots, f_{i-1}^s(\vec{a}_{i-1})/\alpha_{i-1}\big)\Big):=f^s_i(\vec{a}_i),
\]
where $\vec{a}_j\in A^{n_j}$ is the
restriction of $\vec{a}$ to the variables in $\vec{x}_j$, $1\leq j\leq i$.
The functions $F_i$ are obviously 
$V([\psi,\varphi])\cup\{\vec{x}_i\}$-determined and hence
\begin{equation}\label{eqdep}
	\mA\models_{X(A^n/\vec{x},\vec{F}/\vec{\alpha})} \bigwedge_{1\leq i\leq m} \dep[{V([\psi,\varphi])
	\cup\{\vec{x}_i\},\alpha_i}].
\end{equation}
Furthermore, if
$s'\in X(A^n/\vec{x},\vec{F}/\vec{\alpha})$, then in fact
\[
	s'=s(\vec{a}_0/\vec{x}_0,\ldots,\vec{a}_m/\vec{x}_m,
	f^s_1(\vec{a}_1)/\alpha_1,\ldots, f^s_m(\vec{a}_m)/\alpha_m),
\]
for some $s\in X$ and $\vec{a}_i\in A^{n_i}$, $i\leq n$, such that $\vec{a}_1,\dots,\vec{a}_n$ is of pattern $\pi$. Hence
\[
\mA,s'\models \vartheta^*,
\]
for each $s'\in X(A^n/\vec{x},\vec{F}/\vec{\alpha})$. Thus, by induction hypothesis,
\[
\mA\models_{X(A^n/\vec{x},\vec{F}/\vec{\alpha})} \vartheta.
\]
By this and \eqref{eqdep}, we conclude that \eqref{trans-eq} holds and thus that $\mA\models_X\psi$.

The cases (iii) -- (vi) are trivial.
\vspace{1mm}

For the direction $\fopocp \leq \bbd$, let $\varphi\in\fopocp$ be a sentence.
Without loss of generality, we may assume that each variable quantified in $\varphi$ is quantified exactly once.
We define recursively a translation $\psi\mapsto\psi^+$ for all subformulae $\psi$
of $\varphi$ as follows.
If $\psi$ is a literal, we define that $\psi \mapsto \psi.$ For first-order connectives and quantifiers, we define that
\begin{eqnarray*}
(\vartheta\land\eta) &\mapsto& (\vartheta^+\land\eta^+),\\
(\vartheta\lor\eta) &\mapsto& (\vartheta^+\lor\eta^+),\\
\exists x\eta &\mapsto& \exists x\eta^+,\\
\forall x\eta &\mapsto& \forall x\eta^+.
\end{eqnarray*}
Finally, if $\psi=N_\pi\,\vec{x}_1\alpha_1\ldots\vec{x}_m\alpha_m\, \vartheta$,
we define that
\[
	\psi^+:=\forall \vec{x}\,\exists\alpha_1\ldots\exists\alpha_m
	\Big( \big(\bigwedge_{1\le i\le m}\dep[{V([\psi,\varphi])\cup\{\vec{x}_i\},\alpha_i}]
	 \big)\land\vartheta^+\Big),
\]
where $\vec{x}$ is a tuple of exactly those variables that are in at least one of the tuples $\vec{x}_1,\ldots,\vec{x}_m$. Clearly $\varphi^+$ is a $\bbd$ sentence.

It is now easy to prove by induction that for every subformula $\psi$
of $\varphi$
\[
	\mA\models_X\psi^+\;\Leftrightarrow\;\mA,s\models\psi\text{ for all $s\in X$}
\]
holds for every model $\mA$ and team $X$ on $\mA$ such that $\dom(X) =\fr(\psi^+)$. The proof is completely analogous to the inductive proof of the direction $\bbd\leq\fopocp$ shown earlier.
%
%
\end{proof}
By using the same methods as in the proofs of Propositions \ref{Qnormalform}, \ref{dep_nf} and Theorem \ref{bbd-fopocp}, we obtain the following theorem.
\begin{theorem}\label{rbd-pocfo}
$\rbd \equiv \pocfo$ and $\abd \equiv \pocqf$.
\end{theorem}
\begin{proof}
The inclusions $\pocfo\leq\rbd$ and $\pocqf\leq\abd$ follow directly by the translation $\varphi\mapsto \varphi^+$ defined in the proof of Theorem~\ref{bbd-fopocp}.

For the inclusion $\abd \leq \pocqf$, notice that the methods used in the proofs of Propositions \ref{Qnormalform} and \ref{dep_nf} produce, for each $\abd$-sentence $\varphi$, an equivalent $\abd$-sentence of the form
\[
\forall \vec{x}\,\exists \vec{\alpha}\, \psi,
\]
where $\psi$ is a quantifier-free formula. Now, from the translation $\varphi\mapsto \varphi^*$, defined in the proof of Theorem~\ref{bbd-fopocp}, it follows that the sentence $(\forall \vec{x}\,\exists \vec{\alpha}\, \psi)^*$ is a $\pocqf$-sentence equivalent to $\forall \vec{x}\,\exists \vec{\alpha}\, \psi$, and hence equivalent to $\varphi$.

We still need to show that $\rbd \leq \pocfo$. Recall that $\rbd$ is the syntactic fragment of $\bd$ in which there is no Boolean dependence atoms in scope of existential first-order quantifiers. We will first establish that for each $\rbd$-sentence $\varphi$ there exists an equivalent $\rbd$-sentence $\varphi^-$ in which there is no quantification of Boolean variables in the scope of existential first-order quantifiers. Without loss of generality, we consider only models of cardinality at least $2$. For each $\bd$ formula $\exists \alpha\psi$ without Boolean dependence atoms, we define that
\[
(\exists \alpha\psi)' \dfn \exists x\exists y \psi\big((x=y)/\alpha\big),
\]
where $x$ and $y$ are fresh first-order quantifiers not occurring in $\psi$ and the formula $\psi\big((x=y)/\alpha\big)$ is the formula obtained from $\psi$ by substituting each free occurrence of $\alpha$ in $\psi$ by $x=y$. Clearly, the formulae $\exists \alpha\psi$ and $(\exists \alpha\psi)'$ are equivalent in the class of all structures of cardinality at least $2$. Hence, by Theorem \ref{substitution_bd}, we conclude that for each $\rbd$-sentence $\varphi$ there exists an equivalent $\rbd$-sentence $\varphi^-$ in which there is no quantification of Boolean variables in the scope of existential first-order quantifiers.
When the procedures used in the proofs of Propositions \ref{Qnormalform} and \ref{dep_nf} are applied to $\varphi^-$, an equivalent sentence of the form
\[
\forall \vec{x}\,\exists \vec{\alpha}\, \psi,
\]
where $\psi$ is a first-order formula, is obtained. Now from the translation $\varphi\mapsto \varphi^*$ defined in the proof of Theorem~\ref{bbd-fopocp}, it follows that the sentence $(\forall \vec{x}\,\exists \vec{\alpha}\, \psi)^*$ is an $\pocfo$-sentence equivalent to $\varphi^-$ and hence equivalent to $\varphi$.
\end{proof}
\begin{corollary}\label{bbd-0-1-law}
$\bbd$, $\rbd$ and $\abd$ have the zero-one law.
\end{corollary}
\begin{proof}
By Theorem \ref{0-1-law}, $\fopocp$ has the zero-one law. Therefore, by Theorem \ref{bbd-fopocp} $\bbd$ has the zero-one law. Hence the fragments $\rbd$ and $\abd$ of $\bbd$ have the zero-one law.
\end{proof}
%
%
%
%
\begin{theorem}\label{bd-d}
$\bd \equiv \df$.
\end{theorem}
\begin{proof}
$\bd \leq \df$ holds by Proposition \ref{logic-inclusions}.
For the other direction we will give a translation $\varphi\mapsto\varphi^*$ from sentences of dependence logic to sentences of Boolean dependence logic. Let $\varphi$ be an arbitrary $\D$-sentence in the normal form for $\D$ given in \cite[p.~98]{va07}, i.e.,
\[
\varphi := \forall \vec{x} \, \exists \vec{y} \, \big(\bigwedge_{1\leq i\leq n} \dep[\vec{x}_i,y_i]\; \wedge\; \psi\big),
\]
where $\psi$ is a quantifier free first-order formula, $\vec{x}_i$ is a vector of variables from $\vec{x}$ and $y_i$ a variable from $\vec{y}$, $1\leq i\leq n$. The translation $\varphi^*$ of $\varphi$ is the $\bd$-sentence
\[
\forall \vec{x} \, \exists \vec{y}\, \Big(\psi\wedge \forall \vec{z} \, \exists \vec{\alpha} \,
\bigwedge_{1\leq i\leq n} \big(\dep[\vec{x}_i,z_i,\alpha_i]\wedge (z_i=y_i \leftrightarrow \alpha_i) \big) \Big),
\]
where $\vec{z}$ and $\vec{\alpha}$ are tuples of fresh variables of length $n$, and $z_i$ and $\alpha_i$ are variables from the corresponding tuple, $1\leq i\leq m$. We will show that for every model $\mA$
\[
\mA\models_{\{\emptyset\}} \varphi \quad\text{ iff }\quad \mA\models_{\{\emptyset\}} \varphi^*.
\]

Assume first that $\mA\models_{\{\emptyset\}} \varphi$. Hence
\[
\mA\models_X \bigwedge_{1\leq i\leq n} \dep[\vec{x}_i,y_i]\; \wedge\; \psi,
\]
for some team $X$ that can be obtained from $\{\emptyset\}$ by evaluating the quantifier prefix of $\varphi$, i.e., $X= \{\emptyset\}(A^{\lvert\vec{x}\rvert}/\vec{x}, \vec{G}/\vec{y}\,)$ for some functions
\[
G_i: \{\emptyset\}(A^{\lvert\vec{x}\rvert}/\vec{x}, G_1/y_i,\dots,G_{i-1}/y_{i-1})\to A,
\]
$1\leq i\leq \lvert y\rvert$. Now since
\[
\mA\models_X \bigwedge_{1\leq i\leq n} \dep[\vec{x}_i,y_i],
\]
there exists, for every $1\leq i\leq n$, a function
\[
F_i:A^{\lvert \vec{x}_i\rvert}\rightarrow A
\]
that maps the values of the variables $\vec{x}_i$ to the value of the variable $y_i$ in the team $X$.
Let $Y$ denote the team
\[
X\big(A^{\lvert \vec{z}\rvert}/\vec{z},H_1/\alpha_1, \dots, H_n/\alpha_n\big),
\]
where $H_i:X(A^{\lvert \vec{z}\rvert}/\vec{z}, H_1/\alpha_1,\dots H_{i-1}/\alpha_{i-1})\rightarrow \{\bot,\top\}$ is obtained from $F_i$ in the obvious way, i.e., such that that
\begin{align*}
H_i(s) \dfn \begin{cases}
\top &\text{ if } F_i\big(s(\vec{x}_i)\big) = s(z_i)\\
\bot &\text{ if } F_i\big(s(\vec{x}_i)\big) \neq s(z_i).
\end{cases}
\end{align*}
Notice that
\[
\mA\models_Y \bigwedge_{1\leq i\leq n} \big(\dep[\vec{x}_i,z_i,\alpha_i]\wedge (z_i=y_i \leftrightarrow \alpha_i) \big).
\]
Hence we have that
\[
\mA\models_X\forall \vec{z} \, \exists \vec{\alpha} \,
\bigwedge_{1\leq i\leq n} \big(\dep[\vec{x}_i,z_i,\alpha_i]\wedge (z_i=y_i \leftrightarrow \alpha_i) \big).
\]
Therefore, since
$\mA\models_X \psi$,
and since $X$ was obtained from $\{\emptyset\}$ by evaluating the quantifier prefix $\forall\vec{x}\,\exists\vec{y}$, we conclude that
$\mA\models_{\{\emptyset\}} \varphi^*.$

Assume then that
$\mA\models_{\{\emptyset\}}\varphi^*$
holds. Hence
\[
\mA\models_X \psi\wedge \forall \vec{z} \, \exists \vec{\alpha} \,
\bigwedge_{1\leq i\leq n} \big(\dep[\vec{x}_i,z_i,\alpha_i]\wedge (z_i=y_i \leftrightarrow \alpha_i) \big),
\]
for some team $X$ that can be obtained from $\{\emptyset\}$ by evaluating the quantifier prefix of $\varphi^*$. Furthermore
\[
\mA\models_Y\bigwedge_{1\leq i\leq n} \big(\dep[\vec{x}_i,z_i,\alpha_i]\wedge (z_i=y_i \leftrightarrow \alpha_i) \big),
\]
for some team $Y$ that can be obtained from $X$ by evaluating the quantifiers $\forall \vec{z} \, \exists \vec{\alpha}$.
We will show that for each $i$, $1\leq i\leq n$,
\[
\mA\models_X\dep[\vec{x}_i,y_i].
\]
This together with the fact that $\mA\models_X \psi$ is enough to prove that $\mA\models_{\{\emptyset\}}\varphi$.

Fix $i$, $1\leq i\leq n$, and let $s,t\in X$ be any two assignments such that $s(\vec{x}_i)=t(\vec{x}_i)$.
Clearly there exist assignments $s',t'\in Y$ such that
\[
s'\upharpoonright \dom(X)\;=\;s,
\quad\quad
t'\upharpoonright \dom(X)\;=\;t
\quad\text{ and }\quad s'(z_i)\;=\;t'(z_i)\;=\;s(y_i).
\]
Now since
\[
\mA\models_Y z_i=y_i \leftrightarrow \alpha_i,
\]
we have that $s'(\alpha_i)=1$. Furthermore, since
\[
\mA\models_Y\dep[\vec{x}_i,z_i,\alpha_i],
\]
we have that $s'(\alpha)=t'(\alpha)$. Hence $t'(\alpha)=1$ and furthermore $t'(y_i)=t'(z_i)$. Thus, we have that
\[
t(y_i)=t'(y_i)=t'(z_i)=s(y_i),
\]
and can conclude that $t(y_i)=s(y_i)$. 
Therefore
\[
\mA\models_X \dep[\vec{x}_i,y_i].
\]
\end{proof}
By Theorem \ref{d equiv if equiv eso}
we know that $\dl\equiv \eso$. Hence we obtain the following corollary.
\begin{corollary}\label{bd-eso}
$\bd\equiv \eso$.
\end{corollary}

It is well-known that $\eso$ does not have zero-one law. For example, it is easy to write
a sentence $\psi$ of $\eso$ which says that the domain of a model has even cardinality;
clearly the limit probability $\mu(\psi)$ does not exist.

\begin{corollary}\label{bd-0-1-law}
$\bd$ does not have the zero-one law.
\end{corollary}
\section{Hierarchy of expressive power}\label{separations}
%
In Section \ref{equiv} we showed that the expressive power of the fragments \bbd, \rbd and \abd of Boolean dependence logic coincide with the expressive power of the fragments \fopocp, \pocfo and \pocqf of \fopoc, respectively. In this section we show that the fragments \bbd, \rbd and \abd of Boolean dependence logic form a hierarchy with respect to expressive power. We show that
\[
\pocqf < \pocfo < \fopocp.
\]
and hence that
\[
\abd < \rbd < \bbd.
\]
Moreover, we establish that $\bbd < \bd$ and that $\bd\not\leq\fopoc$.
\begin{lemma}\label{abd closed}
Let $\mA$ be a model, $\mB$ a submodel of $\mA$ and $\varphi$ a sentence of \pocqf. If $\mA\models \varphi$ then $\mB\models \varphi$.
\end{lemma}
\begin{proof}
By Corollary \ref{abd-pocqf}, $\pocqf$ has the same expressive power as strict $\NP$. For strict $\NP$ the claim follows from \cite[Lemma~1.2]{barwise:applications}.
\end{proof}

\begin{proposition}\label{pocqf-pocfo}
$\pocqf < \pocfo$.
\end{proposition}
\begin{proof}
Clearly $\pocqf\leq \pocfo$. By Lemma \ref{abd closed}, the truth of a \pocqf-sentence is preserved from models to its submodels. Hence it is enough to give a $\pocfo$-sentence which is not preserved under taking submodels. Clearly
\[
\exists x \exists y \, \neg x = y
\]
is such a sentence.
\end{proof}
Since, by Theorem \ref{rbd-pocfo}, $\abd\equiv\pocqf$ and $\rbd\equiv \pocfo$, the following result follows from Proposition \ref{pocqf-pocfo}.

\begin{corollary}\label{abd-rbd}
$\abd < \rbd$.
\end{corollary}

\begin{definition}\label{equiv_L}
Let $\LL$ be a logic (or a fragment of a logic), $\tau$ a vocabulary, and $\mA$ and $\mB$ first-order 
structures over $\tau$. We write that
$\mA\simp_{\LL}\mB$,
if the implication 
\[
\mA\models \varphi\Rightarrow\mB\models \varphi
\]
holds for every sentence $\varphi \in \LL$.
\end{definition}

Let $N_\pi$ be a partially-ordered connective. By $\mnfor$, we denote the set of all sentences in $\pocfo$ which are of the
form
\[
N_\pi\vec x_1\alpha_1\dots \vec x_m \alpha_m\,\varphi,
\]
where
$\varphi$ is a first-order formula with quantifier rank at most $r$.
We will next define an Ehrenfeucht-Fra\"{\i}ss\'{e} game that captures the truth preservation
relation $\mA\simp_{\mnfor}\mB$. This game is a straightforward modification of the 
corresponding game for \pocsvfo by Sevenster and
Tulenheimo \cite{setu06c}, which in turn is based on the game for \fopocsv by Sandu and V\"a\"an\"anen \cite{sava92}. 

\begin{definition}\label{ef game rbd}
Let $\mA$ and $\mB$ be first-order structures over a vocabulary $\tau$ and 
$r\geq 0$. Let $\pi=(n_1,\dots,n_m,E)$ be a pattern. The \emph{$\mnfor$-EF game} 
$\efgmnr(\mA,\mB)$ 
on $\mA$ and 
$\mB$ is played by two players, Spoiler and Duplicator. The game has two phases.
\vspace{1mm}

\noindent {\bf Phase 1:}
\begin{itemize}
\vspace{-2.3mm}
\item Spoiler picks a function $f_i: A^{n_i}\to\{\bot,\top\}$, for each $i$, $1\le i\le m$. 
\item Duplicator answers by choosing a function $g_i: B^{n_i}\to\{\bot,\top\}$, for each $i$, $1\le i\le m$. 
\item Spoiler chooses tuples $\vec b_i\in B^{n_i}$, $1\le i\le m$, such that $\vec b_1\ldots\vec b_m$
is of pattern $\pi$. 
\item Duplicator answers by choosing tuples $\vec a_i\in A^{n_i}$, $1\le i\le m$, such that
$\vec a_1\ldots\vec a_m$ is of pattern $\pi$, and  $f_i(\vec a_i)=g_i( \vec b_i)$
for each $1\leq i\leq m$. If there are no such tuples 
$\vec a_1,\ldots,\vec a_m$, then Duplicator loses the play 
of the game.
\end{itemize}
\vspace{-1.4mm}
\noindent {\bf Phase 2:}
\begin{itemize}
\vspace{-2.3mm}
\item Spoiler and Duplicator play the usual first-order EF-game of $r$ rounds
on the structures $(\mA,\vec a_1\ldots\vec a_m)$ and $(\mB,\vec b_1\ldots\vec b_m)$:
On each round $j$, $1\leq j\leq r$, Spoiler picks an element $c_j\in A$ (or $d_j\in B$), and
Duplicator answer by choosing an element $d_j\in B$ (or $c_j\in A$, respectively).
\end{itemize}
Duplicator wins the play of the game if and only if the mapping
\[
\vec a_1\ldots\vec a_m c_1\ldots c_r\mapsto \vec b_1\ldots\vec b_m d_1\ldots d_r
\]
is a partial isomorphism from $\mA$ to $\mB$.
We say that Duplicator has a winning strategy in the game if and only she has a 
systematic way of answering all possible moves of Spoiler such that using it she 
always wins the play. 
\end{definition}

We show next that the game $\efgmnr$ can be used for studying the truth preservation
relation $\simp_{\mnfor}$. 
This result is essentially the same as Proposition 12 in \cite{setu06c},
which in turn is a special case of Proposition 7 of \cite{sava92}. 

\begin{proposition}\label{ef char}
Let $\mA$ and $\mB$ be $\tau$-structures,
$\pi$ a pattern,  and $r\geq 0$. If Duplicator has a winning strategy in 
the game $\efgmnr(\mA,\mB)$ then $\mA\simp_{\mnfor}\mB$.
\end{proposition}
\begin{proof}
Assume that Duplicator has a winning strategy in $\efgmnr(\mA,\mB)$.
To prove $\mA\simp_{\mnfor}\mB$, assume that
\[
\varphi=N_\pi\,\vec x_1\alpha_1\dots \vec x_m \alpha_m\,\psi
\]
is a sentence
of $\mnfor$ such that $\mA\models\varphi$. We need to show that $\mB\models\varphi$.  Now since $\mA\models\varphi$, there exists functions
\[
f_i: A^{n_i}\to \{\bot,\top\},
\]
$1\leq i\leq m$, such that if $\vec{a}_1\ldots\vec{a}_m$, where $\vec a_j\in A^{n_j}$ for $1\leq j\leq m$,
is of pattern $\pi$ then
\[
\mA\models_{s[\vec{a},\vec{f}]}\,\psi,
\]
where $s[\vec{a},\vec{f}]$ denotes the assignment that maps $\vec{x}_i$ to $\vec{a}_i$ and
$\alpha_i$ to $f_i(\vec{a}_i)$, for $1\le i\le m$.
%

Assume that Spoiler chooses the functions $f_1,\dots,f_m$ as his first move in phase 1 of the game $\efgmnr(\mA,\mB)$. Let
\[
g_i: B^{n_i}\to\{\bot,\top\},
\]
$1\le i\le m$, be the answer given by the winning strategy of Duplicator. Remember that by $s[\vec{b},\vec{g}]$ we mean the assignment that maps $\vec{x}_i$ to $\vec{b}_i$ and
$\alpha_i$ to $g_i(\vec{b}_i)$, for $1\le i\le m$. To show $\mB\models\varphi$,
it suffices to show that
\[
\mB\models_{s[\vec{b},\vec{g}]}\psi,
\]
for all tuples 
$\vec{b}_1\ldots\vec{b}_m$ of pattern $\pi$ such that $\vec{b}_i\in B^{n_i}$, for each~$i\leq m$.
Thus, let $\vec b_i\in B^{n_i}$, $1\le i\le m$, be arbitrary tuples such that 
$\vec b_1\ldots\vec b_m$ is of pattern~$\pi$.
Let $\vec{a}_j\in A^{n_j}$, $1\leq j\leq m$, be the answer given by 
the winning strategy of Duplicator when the second move of Spoiler is $\vec b_1,\ldots,\vec b_m$.
By the definition of the game $\efgmnr$, Duplicator then has a winning strategy in the first-order 
EF game with $r$ rounds between the structures
\[
(\mA,\vec a_1,\ldots,\vec a_m)  \quad\text{ and }\quad  (\mB,\vec b_1,\ldots,\vec b_m).
\]
By the standard EF theorem, it follows that 
$(\mA,\vec a_1,\ldots,\vec a_m)$ and $(\mB,\vec b_1,\ldots,\vec b_m)$ satisfy the
same $\fo_r$-sentences. Note further that
\[
f_1(\vec a_1)=g_1( \vec b_1),\ldots, f_m(\vec a_m)=g_m( \vec b_m),
\]
by the condition
governing the choice of $\vec b_1,\ldots,\vec b_m$, whence
\[
\big(\mA,\vec a_1,\ldots,\vec a_m, f_1(\vec a_1),\ldots,f_m(\vec a_m)\big) \quad\text{ and }\quad 
\big(\mB,\vec b_1,\ldots,\vec b_m, g_1(\vec b_1),\ldots,g_m(\vec b_m)\big)
\]
are 
equivalent 
with respect to all sentences of $\fo_r$ extended with Boolean variables. In particular,
since $\mA\models_{s[\vec a,\vec f]}\psi$, we have $\mB\models_{s[\vec b,\vec g]}\psi$, as desired.
\end{proof}
%
\begin{corollary}\label{pocfo def}
Let $\mathcal{K}$ be a class of $\tau$-structures. If for every pattern $\pi$ and every
$r\geq 0$ there exist $\tau$-structures $\mA$ and $\mB$
such that $\mA\in \mathcal{K}$, $\mB\not\in \mathcal{K}$ and Duplicator has a winning strategy in the game
$\efgmnr(\mA,\mB)$, then $K$ is not definable in $\pocfo$.
\end{corollary}
\begin{proof}
We prove the contraposition of the claim. Thus,
assume that $K$ is definable in $\pocfo$. Then there is a sentence $\varphi$
of the form $N_\pi\vec x_1\alpha_1\dots \vec x_m\alpha_m\,\psi$ such that
\[
\mathcal{K}=\{\mA\in \mathrm{Str}(\tau)\mid\mA\models\varphi\}.
\]
Let $r$ be the quantifier rank of $\psi$. Then 
by Proposition \ref{ef char}, there are no structures $\mA\in \mathcal{K}$ and $\mB\not\in \mathcal{K}$ such that Duplicator
has a winning strategy in the game $\efgmnr(\mA,\mB)$.  
\end{proof}

\begin{theorem}\label{rbd-bbd}
$\rbd < \bbd$
\end{theorem}
\begin{proof}
By Proposition \ref{logic-inclusions}, $\rbd\leq \bbd$. For the strict inclusion, we show that non-connectivity of graphs is definable in $\bbd$, but not in $\rbd$. Let $\mathcal{K}$ denote the class non-connected graphs.
Note first that a graph $\mA=(A,E^\mA)$ is not connected if and only if there is a 
subset $U\subseteq A$ such that $U$ and $A\setminus U$ are nonempty, and
there are no edges $(a,b)\in E^\mA$ between  $U$ and $A\setminus U$. This can 
be expressed by the $\bbd$-sentence
%
\[
\begin{array}{lll}
   &\exists u\exists v\forall x\forall y \exists\alpha\exists\beta \big(\\
     &\quad   \dep[x,\alpha] \land \dep[y,\beta] &\arraycomment{There are two relations\dots}\\
   &\quad \land  (x=y\to (\alpha\leftrightarrow\beta)) &\arraycomment{\dots which are equal\dots}\\
   &\quad \land  (x=u\to \alpha) \land (x=v\to \lnot\alpha) &\arraycomment{\dots and contain $u$ but not $v$.}\\
   &\quad \land  (\alpha\land\lnot\beta\to \lnot Exy)\, \big). &\arraycomment{\parbox[t]{5.5cm}{If $x$ is in the relations but $y$ is not then there is no edge between $x$ and $y$.}}
\end{array}
\]
%

We use Corollary~\ref{pocfo def} to prove that non-connectivity is not definable
in $\pocfo$. By Theorem~\ref{rbd-pocfo}, it then follows that non-connectivity is not definable
in $\rbd$. 
Let us fix  the pattern $\pi$ and the number of rounds $r\geq 0$, and consider the 
game $\efgmnr$. Let $\mA=(A,E^\mA)$ and $\mB=(B,E^\mB)$ be
the graphs such that
\begin{itemize}
\item $B=\{u_1,\ldots, u_k\}$, 
and $A=B\cup \{v_1,\ldots,v_k\}$,
\item $E^\mB=\{(u_i,u_j)\in B^2 \mid |i-j|=1\}\cup \{(u_1,u_k),(u_k,u_1)\}$, 
\item $E^\mA=E^\mB\cup\{(v_i,v_j)\in A^2\mid |i-j|=1\}\cup \{(v_1,v_k),(v_k,v_1)\}$.
\end{itemize}
Thus, $\mB$ is a cycle of length $k$, and $\mA$ is the disjoint union of
two cycles of length $k$. In particular, $\mA\in \mathcal{K}$ and $\mB\not\in \mathcal{K}$. We will show that if $k$ is large enough, then Duplicator has a winning 
strategy in the game $\efgmnr(\mA,\mB)$. By Corollary~\ref{pocfo def} it then follows that $\mathcal{K}$, i.e., non-connectivity of graphs, is not definable in $\pocfo$.

Let
\[
f_i:A^{n_i}\to\{\bot,\top\},\, 1\le i\le m,
\]
be the functions that Spoiler picks on his first move in the game. 
Duplicator will then answer by picking the functions
\[
g_i: B^{n_i}\to\{\bot,\top\}, 1\le i\le m,
\]
where
$g_i:=f_i\upharpoonright B$, for each $1\leq i\leq m$. 
For his next move, Spoiler picks a tuple 
\[
\vec b_i\in B^{n_i}, \text{ for each $i$, }1\le i\le m,
\]
such that $\vec b_1\ldots\vec b_m$
is of pattern $\pi$. Now, Duplicator can simply answer by choosing
the same tuples: let
\[
\vec a_i\dfn\vec b_i \text{, for each  $1\leq i\leq m$.}
\]
Clearly
the requirement $f_i(\vec a_i)=g_i(\vec b_i)$ is then satisfied.
The game continues after this as the first-order EF-game with $r$ rounds
on the structures $(\mA,\vec a_1\ldots\vec a_m)$ and $(\mB,\vec b_1\ldots\vec b_m)$.
Since the mapping
\[
\vec a_1\ldots\vec a_m\mapsto\vec b_1\ldots\vec b_m
\]
respects distances between nodes in the graphs $\mA$ and $\mB$, a standard 
argument shows that Duplicator has a winning strategy in the rest of the game,
provided that $k$ is big enough. 
%
%
\end{proof}
By Theorems \ref{bbd-fopocp} and \ref{rbd-pocfo}, $\bbd\equiv\fopocp$ and $\rbd\equiv\pocfo$. Hence we obtain the following corollary.
\begin{corollary}
$\pocfo < \fopocp$.
\end{corollary}
As a byproduct of the results concerning the zero-one law, we obtain the following results concerning expressive power.
\begin{proposition}\label{rbd-bd}
$\bbd < \bd$, and moreover $\bd\not\leq\fopoc$.
\end{proposition}
\begin{proof}
Clearly $\bbd\leq\bd$. By Corollary \ref{0-1-law} and Corollary \ref{bbd-0-1-law}, $\fopoc$ and $\bbd$ have the zero-one law. By Corollary \ref{bd-0-1-law}, $\bd$ does not have the zero-one law. Therefore $\bbd<\bd$ and $\bd\not\leq\fopoc$.
\end{proof}

\section{Conclusion}
In this article we defined a new variant of dependence logic called Boolean dependence logic. Boolean dependence logic is an extension of first-order logic with dependence atoms of the form $\dep[\vec{x},\alpha]$, where $\vec{x}$ is a tuple of first-order variables and $\alpha$ is a Boolean variable. We also introduced a notational variant of partially-ordered connectives based on the narrow Henkin quantifiers of Blass and Gurevich \cite{blassgure}. We showed that the expressive power of Boolean dependence logic and dependence logic coincide. We defined natural syntactic fragments of Boolean dependence logic and proved that the expressive power of these fragments coincide with corresponding logics based on partially-ordered connectives. More formally, we showed that
\[
\bbd\equiv\fopocp, \rbd\equiv\pocfo\ \text{ and } \abd\equiv\pocqf.
\]
Moreover, we proved that the fragments of Boolean dependence logic form a strict hierarchy in terms of expressive power, i.e., we showed that
\[
\abd < \rbd < \bbd < \bd.
\]
Therefore, we also showed that
\[
\pocqf < \pocfo < \fopocp.
\]
In addition, we obtained that $\bd$ does not have the zero-one law, whereas the logics below $\bbd$ and $\fopoc$ have the zero-one law. Therefore we obtained that $\bd\not\leq \fopoc$. 

\bibliography{tampere1}

\end{document}